\documentclass[11pt]{article}

\usepackage{fancyhdr}
\usepackage{amsfonts}
\usepackage{amsmath}
\usepackage{amssymb}
\usepackage{amsthm}
\usepackage{tikz}
\usepackage{amscd}
\usepackage{amstext}

\usepackage{cases}

\usepackage{accents}
\usepackage{tikz-cd}

\usepackage{mathrsfs}

\usepackage{authblk}

\usepackage{graphics}
\usepackage{graphicx}

\usepackage{graphics}
\usepackage{graphicx}

\newlength{\fixboxwidth}
\newcommand{\re}{\mathbb{R}}\newcommand{\N}{\mathbb{N}}
\newcommand{\zz}{\mathbb{Z}}

\newcommand{\Z}{{\zz}^d}

\newcommand{\R}{{\re}^d}

\newcommand{\cs}{{\mathcal S}}

\newcommand{\cl}{{\mathcal L}}

\newcommand{\ce}{{\mathcal E}}

\newcommand{\cf}{{\mathcal F}}
\newcommand{\cfi}{{\cf}^{-1}}

\newcommand{\supp}{{\rm supp \, }}

\newcommand{\mix}{{\rm mix}}

\newcommand{\bproof}{\begin{proof}}
\newcommand{\eproof}{\end{proof}}

\newcommand{\be}{\begin{equation}}
\newcommand{\ee}{\end{equation}}
\newcommand{\beq}{\begin{eqnarray}}
\newcommand{\beqq}{\begin{eqnarray*}}
\newcommand{\eeq}{\end{eqnarray}}
\newcommand{\eeqq}{\end{eqnarray*}}

\numberwithin{equation}{section}

\newtheorem{theorem}{Theorem}[section]
\newtheorem{definition}[theorem]{Definition}
\newtheorem{corollary}[theorem]{Corollary}
\newtheorem{lemma}[theorem]{Lemma}
\newtheorem{proposition}[theorem]{Proposition}
\newtheorem{remark}[theorem]{Remark}

\begin{document}

%%%%%%%%%%%%%%%%%%%%%%%%%%%%%%%%%%%%%%%%%%%%%%%%%%%%%%%%%%%%%%%%%%%%%%%%%%%%%%%%%%%%%
%%%%%%%%%%%%%%%%%%%%%%%%%%%%%%%%%%%%%%%%%%%%%%%%%%%%%%%%%%%%%%%%%%%%%%%%%%%%%%%%%%%%%

\title{Gelfand Numbers of Embeddings of Mixed Besov Spaces}

\author[a,b]{Van Kien Nguyen\thanks{E-mail: kien.nguyen@uni-jena.de \& kiennv@utc.edu.vn}}
%email{kien.nguyen@uni-jena.de \quad \& \quad  winfried.sickel@uni-jena.de}
\affil[a]{Friedrich-Schiller-University Jena, Ernst-Abbe-Platz 2, 07737 Jena, Germany}
\affil[b]{University of Transport and Communications, Dong Da, Hanoi, Vietnam}
%%%%%%%%%%%%%%%%%%%%%%%%%%%%%%%%%%%%%%%%%%%%%%%%%%%%%%%%%%%%%%%%%%%%%%%%%%%%%%%%%%%%%
%%%%%%%%%%%%%%%%%%%%%%%%%%%%%%%%%%%%%%%%%%%%%%%%%%%%%%%%%%%%%%%%%%%%%%%%%%%%%%%%%%%%%

\date{\today}

\maketitle

\begin{abstract}
Gelfand numbers represent a  measure for the information complexity which is given by the number of information needed to approximate functions in a subset of a normed space with an error less than $\varepsilon$. More precisely, Gelfand numbers coincide up to  the factor 2 with the minimal error $ e^{\rm wor}(n,\Lambda^{\rm all})$ which describes the error of the optimal (non-linear) algorithm that is based on $n$ arbitrary linear functionals. This explains the  crucial role of Gelfand numbers in the study of approximation problems. Let $S^t_{p_1,p_1}B((0,1)^d)$ be the Besov spaces with dominating mixed smoothness on $(0,1)^d$. In this paper we consider the problem ${\rm App}: S^t_{p_1,p_1}B((0,1)^d) \to L_{p_2}((0,1)^d)$ and investigate the asymptotic behaviour of Gelfand numbers of this embedding. We shall give the correct order of convergence of Gelfand numbers in almost all cases. In addition we shall compare these results with the known
behaviour of approximation numbers which coincide with $ e^{\rm wor-lin}(n,\Lambda^{\rm all})$ when we only allow linear algorithms. 
\end{abstract}

%&&&&&&&&&&&&&&&&&&&&&&&&&&&&&&&&&&&&&&&&&&&&&&&&&&
%&&&&&&&&&&&&&&&&&&&&&&&&&&&&&&&&&&&&&&&&&&&&&&&&&&

\section{Introduction}\label{intro}

Gelfand and approximation numbers play a crucial role in information-based complexity. Let us first recall some related notions, see \cite{NoWo-1,TrWaWo}. Let $\tilde{F}$ and $G$ be normed spaces of functions defined on the set $D_d\subset \R$.  We consider the linear operator
\be\label{prob}
{\rm App}: F \to G,\qquad\ \ {\rm App}(f)=f
\ee 
where $F$ is a subset of $\tilde{F}$, such as the unit ball of $\tilde{F}$. Our aim consists in computing an approximation
 of $f\in F$. Let $$N(f)=[L_1(f),...,L_n(f)]\in \re^n$$ be the information about $f\in F$ we can use.  Here $L_i\in \Lambda \subset \tilde{F}'$, a subset of the set of all linear, real-valued and continuous  functionals on $\tilde{F}$. We are interested in two different classes  $\Lambda$. First,  $\Lambda=\Lambda^{\rm all}=\tilde{F}'$. Second,  $\Lambda=\Lambda^{\rm std}$, the set of all linear functionals generated by function value, i.e., for some $x\in D_d$ we have 
\beqq
L(f)=f(x),\ \ \text{for all} \ f\in F.
\eeqq 
This type of information is called standard information. To approximate $f\in F$ we use algorithms of the form $A =\varphi \circ N $ where $\varphi: \re^n\to G$ is an arbitrary mapping. Then the worst case error of the algorithm $A$ is given by
\beqq
e^{\rm wor}(n,A)  =\sup_{f\in F}\| f-A(f)|G\| . 
\eeqq
The minimal error of the class $\Lambda$ is defined as
\beqq
e^{\rm wor}(n,\Lambda) =\inf_{A:L_i\in \Lambda, i=1,...,n} e^{\rm wor}(n,A).
\eeqq
In such a situation it is well-known that 
\be \label{wor-gel}
c_n(F,G)\leq e^{\rm wor}(n,\Lambda^{\rm all}) \leq 2 c_n(F,G)
\ee 
for all $n\geq 1$, see \cite[Section 5.4]{TrWaWo}. Here $c_n(F,G)$ is the Gelfand $n$-width of the set $F$ in $G$. The error $e^{\rm wor}(n,\Lambda)$ is inversely related to the information complexity $n^{\rm wor}(\varepsilon,\Lambda)$   which is given by
\beqq
n^{\rm wor}(\varepsilon,\Lambda) =\min\{ n: \text{  there exists }\ A\ \text{with}\  e^{\rm wor}(n,A) \leq \varepsilon\}.
\eeqq
The number $n^{\rm wor}(\varepsilon,\Lambda) $ shows that to solve the problem \eqref{prob} within an error of $\varepsilon>0$, we need $n$ information operations in the class $\Lambda$.  If we only allow linear algorithms $\varphi: \re^n \to G$ then we get  
\beqq
e^{\rm wor-lin}(n,\Lambda^{\rm all}) =\inf_{A:L_i\in \Lambda, i=1,...,n\atop \varphi\ \text{is linear}} e^{\rm wor}(n,A).
\eeqq
These are the approximation numbers of the embedding ${\rm App}: \tilde{F}\to G$, sometimes also called linear widths.
We wish to emphasize that linear algorithms are not always optimal, see, e.g., \cite[Section 4.2]{NoWo-1}. This explains the importance of Gelfand numbers in the study of the information complexity of the class $\Lambda^{\rm all}$.  

There is an increasing interest in information-based complexity and high-dimensional approximation  in the context of function spaces with dominating mixed smoothness. The reason for this is clear, function spaces with dominating mixed smoothness are much smaller than their isotropic counterpart (with the same smoothness). There is a realistic hope that one can approximate functions from these classes for larger dimension than in case of isotropic spaces. Let us mention that there exist a number of problems in finance and quantum chemistry modeled on function spaces with dominating mixed smoothness, see, e.g., \cite{Glas-04} and \cite{Yse-10}.

Let  $\Omega$ be the unit cube of $\R$, i.e., $\Omega=(0,1)^d$. The purpose of the present paper is to study the order of convergence of Gelfand numbers of the embedding $${\rm App}: S^t_{p_1,p_1}B(\Omega) \to L_{p_2}(\Omega),$$ $1<p_1,p_2<\infty$ and $t>\max(\frac{1}{p_1}-\frac{1}{p_2},0)$. Here  $S^t_{p_1,p_1}B(\Omega)$  denotes  the Besov space with dominating mixed smoothness on $\Omega$. A particular interesting special case is given by $p_1=2$. Then $ S^t_{2,2}B(\Omega)$ coincides with $H_{\mix}^t(\Omega)$, a space, which attracts a lot of attention recently in numerical analysis. Note that if $t=m\in \N$, then these spaces can be simply described as the collection of all $f\in L_2(\Omega)$ such that all distributional derivatives $D^{\alpha}f$ with $\alpha=(\alpha_1,...,\alpha_d)\in \N_0^d$ and $\max_{i=1,...,d} |\alpha_i|\leq m$ belong to $L_2(\Omega)$. This paper is a continuation of \cite{KiSi}-\cite{Ki16}.
 
The paper is organized as follows. Our main results are discussed in Section \ref{sec-main}. 
Section \ref{besov} is devoted to  the function spaces under consideration.
In Section \ref{proof} we prove the results for sequence spaces 
associated to mixed Besov spaces and transfer them to the level of function spaces.

\noindent
{\bf Notation:} As usual, $\N$ denotes the natural numbers, $\N_0 := \N \cup \{0\}$,
$\zz$ the integers and
$\re$ the real numbers. For a real number $a$ we put $a_+ := \max(a,0)$.
By $[a]$ we denote the integer part of $a$.
If $\bar{j}=(j_1,... , j_d) \in \N_0^d$, then we put
$
|\bar{j}|_1 := j_1 + \ldots \, + j_d\, .
$
If $X$ and $Y$ are two Banach spaces, then the symbol $X\hookrightarrow Y$ indicates that the embedding is continuous. $X'$ denotes the dual space of $X$.
The meaning of $A \lesssim B$ is given by: there exists a constant $c>0$ such that
 $A \le c \,B$. Similarly $\gtrsim$ is defined. The symbol 
$A \asymp B$ will be used as an abbreviation of
$A \lesssim B \lesssim A$.
For a finite set $\nabla$ the symbol $|\nabla|$ 
denotes the cardinality of this set.
Finally, the symbols ${id}, id^*$ and $\rm App$ will be used for identity operators, ${id}, id^*$ mainly connection with sequence spaces and App with function spaces.
The symbol $id_{p_1,p_2}^m$ refers to the identity 
\be\label{idlp}
id_{p_1,p_2}^m:~ \ell_{p_1}^m \to \ell_{p_2}^m\, .
\ee
\section{The main results}\label{sec-main}

 Let $X$, $Y$ be Banach spaces and $T$ be a continuous linear operator from $X$ to $Y$, i.e., $T \in \mathcal L(X,Y)$. The $n$th Gelfand number of $T $ is defined as 
$$ c_n (T) := \inf\Big\{\|\, T J_M^X\, \|: \ {\rm codim\,}(M)< n\Big\},$$
where $J_M^X:M\to X$ refers to the canonical injection of $M$ into $X$. Let $A$ be a subset of $Y$. The Gelfand $n$-width of the set $A$ in $Y$ is given by
\beqq
c_n(A,Y):=\inf_{L_n}\sup_{x\in A \cap L_n}\|x|Y\|
\eeqq
where  the infimum  is  taken over all  subspaces $L_n$ 
of codimension $n$ in $Y$. 
If $T$ is a compact operator then the $(n+1)$th Gelfand number of the operator $T\in \mathcal{L}(X,Y)$ and the Gelfand $n$-width of $T(B_X)$ in $Y$ coincide, see \cite{EL}. Here $B_X$ is the closed unit ball of $X$.

Related to Gelfand numbers are the Kolmogorov, approximation and Weyl numbers. The $n$th Kolmogorov number of the linear operator $T \in \mathcal L(X,Y)$ is defined as 
\beqq
d_n(T)= \inf_{L_{n-1}}\sup_{\|x|X\|\leq 1}\inf_{y\in L_{n-1}}\|Tx-y|Y\|.
\eeqq
Here the outer infimum is taken over all linear subspaces  $L_{n-1}$ of dimension ($n-1$)  in $Y$. The $n$th approximation number of $T$ is defined as
 $$ a_n(T):=\inf\{\|T-A\|: \ A\in \mathcal L(X,Y),\ \ \text{rank} (A)<n\}\, .  $$ 
And the $n$th Weyl number of $T$ is given by
 \beqq
 \begin{split} 
  x_n(T):&=\sup\{a_n(TA):\ A\in \mathcal L(\ell_2,X),\ \|A\|\leq 1\}\,.
  \end{split}
  \eeqq
The inequality
  \be \label{xca}
  x_n(T)\leq c_n(T) \leq a_n(T),
  \ee 
valid for every bounded linear operator $T$, see \cite[Theorem 2.10.1]{Pi-87}, and the relation
\be \label{dual}
c_n(T)=d_n(T')
\ee
if $T$ is a compact operator, see \cite[Theorem 11.7.7]{Pi-80}, are useful tools when dealing with Gelfand numbers. Here $T'$ denotes the dual operator of $T$.

Gelfand numbers, as well as Kolmogorov, approximation and Weyl numbers belong to the class of  $s$-numbers. Here we use the definition of $s$-numbers in \cite[Section 2.2]{Pi-87}. Let $X,Y,X_0,Y_0$ be Banach spaces.
  An $s$-function is a map $s$ assigning to every operator $T\in \mathcal L(X,Y)$ a scalar sequence $\{s_n(T)\}_{n\in \N}$ such that the following conditions are satisfied:
 \begin{enumerate}
 \item[(a)] $\|T\|=s_1(T)\geq s_2(T)\geq...\geq 0 $;
 \item[(b)]$ s_{n+m-1}(S+T)\leq s_n(S)+ s_m(T) $ for all $S\in \mathcal L(X,Y)$ and $m,n\in \N\, $,
 \item[(c)] $s_n(BTA)\leq \|B\| \, \cdot \, s_n(T) \, \cdot \, \|A\|$ for all $A\in \mathcal L(X_0,X)$, $B\in \mathcal L(Y,Y_0)$;
 \item[(d)] $s_n(T)=0$ if $\text{rank}(T)<n$ for all $n\in \N$; 
 \item[(e)] $s_n(id: \ell_2^n\to \ell_2^n)=1$ for all $n\in \N$.
 \end{enumerate}
 \begin{remark}\label{rem-s}\rm (i) In the literature there is some ambiguity concerning the notion of $s$-numbers. There is a different definition of $s$-numbers in which one replaces axiom (b) by a weaker condition, i.e., $ s_{n}(S+T)\leq s_n(S)+ \|T\| $ for all $S,T\in \mathcal L(X,Y)$ and $m,n\in \N$, see \cite[Section 11.1]{Pi-80}. For more details about $s$-numbers and $n$-widths we refer to the monographs of Pietsch \cite[Chapter 11]{Pi-80}, \cite[Chapter 2]{Pi-87} and Pinkus \cite[Chapter 2]{Pin85}.\\
(ii) In the recent comprehensive survey \cite{DTU-16} of Dinh D\~ung, Temlyakov and Ullrich the reader can find the state of the art concerning the behaviour of $s$-numbers for embeddings of function spaces with dominating mixed smoothness into Lebesgue spaces.
 \end{remark}

Our main result reads as follows.
\begin{theorem}\label{main} Let $1< p_1,p_2< \infty $ and $ t>(\frac{1}{p_1}-\frac{1}{p_2})_+$. Then we have
\[
c_n({\rm App}:\ S^t_{p_1,p_1}B(\Omega) \to L_{p_2}(\Omega))\asymp n^{-\alpha}(\log n)^{(d-1)\beta}, \qquad n \ge2\, , 
 \]
where
\begin{enumerate}
\item 
{\makebox[5.5cm][l]{$\alpha=t$, $\beta=t-\frac{1}{p_1}+\frac{1}{2}$} if  \ $ \max(2,p_2)\leq p_1 $ or \big($  p_1, p_2<2$, $t>\frac{1}{2}$\big)};
\item {\makebox[5.5cm][l]{$\alpha=\beta=t-\frac{1}{p_1}+\frac{1}{p_2}$}  if  \ $2\leq p_1\leq p_2 $};
\item 
{\makebox[5.5cm][l]{$\alpha= t-\frac{1}{2}+\frac{1}{p_2}$, $\beta=t-\frac{1}{p_1}+\frac{1}{p_2}$} if   \ 
$p_1< 2\leq p_2 $,\ $t>1-\frac{1}{p_2}$};
\item {\makebox[5.5cm][l]{$\alpha=\frac{p_1'}{2}(t-\frac{1}{p_1}+\frac{1}{p_2})$, $\beta=\frac{2\alpha}
{p_1'}$} if\ \ \big($  p_1\leq 2< p_2    $, $t<1-\frac{1}{p_2}$\big) or \big($     p_1  <p_2\leq 2 $,\ $t< \frac{1/p_1-1/p_2}{2/p_1-1}$\big)}.
\end{enumerate}
\end{theorem}
\begin{remark}\rm 
(i) Let $1<p_1,p_2<\infty$. Then ${\rm App}:\ S^t_{p_1,p_1}B(\Omega) \to L_{p_2}(\Omega)$ is compact if and only if $t> (\frac{1}{p_1}-\frac{1}{p_2})_+$, see \cite[Theorem 3.17]{Vy-06}. Hence the restriction $t>(\frac{1}{p_1}-\frac{1}{p_2})_+$ is natural. This condition guarantees that Gelfand numbers converge to $0$ as 
$n$ tends to infinity.\\
(ii) Gelfand numbers of embeddings $B^t_{p_1,p_1}(\Omega)\to L_{p_2}(\Omega)$ have been investigated by Vybiral \cite{Vy-08}. Here $B^t_{p_1,p_1}(\Omega)$ denotes isotropic Besov space on $\Omega$. There are a few more references where Gelfand numbers of such embeddings in slightly modified situations have been considered, see \cite{ZF-12,ZF-13,ZD-14,ZFH-14}.
\end{remark}
The picture in Theorem \ref{main} is nearly complete except one case plus some limiting situations. The only case which has been left open consists in $1<p_1,p_2<2$ and $ \max( 0, \frac{1/p_1-1/p_2}{2/p_1-1})<t<\frac{1}{2}$.
\begin{proposition}\label{conj} Let $1<p_1,p_2<2$ and $ \max( 0, \frac{1/p_1-1/p_2}{2/p_1-1})<t<\frac{1}{2}$. Then we have
\[
c_n({\rm App}:\ S^t_{p_1,p_1}B(\Omega) \to L_{p_2}(\Omega))\lesssim n^{-t}(\log n)^{(d-1)(2t-\frac{2t}{p_1})}, \qquad n \ge2\, .
 \]
\end{proposition}
We conjecture that the upper estimate given in Proposition \ref{conj} is sharp. Now we turn to extreme cases given by either $p_2=\infty$ or $p_2=1$. Let us recall a result of Temlyakov \cite{Te93}, see also \cite{CKS}.
\begin{proposition}\label{known1}
Let $t>\frac{1}{2}$. Then we have 
\beqq
 c_{n} ({\rm App}: \, S^{t}_{2,2}B (\Omega) \to L_\infty (\Omega)) \asymp
n^{-t+\frac{1}{2}} (\log n)^{(d-1)t} , \qquad n\geq 2\, .
\eeqq
\end{proposition}
\begin{remark}\label{remark}\rm In the literature many times the notation 
$H^{t}_{\text{mix}} (\Omega) $  and $MW^t_2 (\Omega)$ are used instead of $S^{t}_{2,2}B (\Omega)$. In \cite{Te93,CKS} the authors deal with approximation numbers. However, for Banach spaces $Y$ and Hilbert spaces $H$ we always have
\beqq
x_n(T:\ H\to Y)\, =\, c_n(T:\ H\to Y)\, =\, a_n(T:\ H\to Y),
\eeqq
see \cite[Proposition 11.5.2]{Pi-80} and \cite[Proposition 2.4.20]{Pi-87}.
\end{remark}
By using abstract properties of Gelfand numbers, results in Theorem \ref{main} and Proposition \ref{known1} one can derive the following result.
\begin{theorem}\label{ex} {\rm (i)} Let either $2\leq p<\infty$ and $t>0$ or $1<p<2$ and $t>\frac{1}{2}$. Then we have
 \[
 c_n({\rm App}:\ S^t_{p,p}B(\Omega) \to L_{1}(\Omega))\asymp n^{-t}(\log n)^{(d-1)( t-\frac{1}{p}+\frac{1}{2})}, \qquad n \ge2\, .
  \]
{\rm (ii)} Let $1<p\leq 2$ and  $t>1$. Then we have
 \[
 c_n({\rm App}:\ S^t_{p,p}B(\Omega) \to L_{\infty}(\Omega))\asymp n^{-t+\frac{1}{2}}(\log n)^{(d-1)(t-\frac{1}{p}+\frac{1}{2})}, \qquad n \ge2\, .
  \]
\end{theorem}
\begin{remark}\rm Observe that part (ii) in Theorem \ref{ex} is not the limit of part (iii) in Theorem \ref{main} when $p_2 \to \infty$. More exactly, there is a jump of order $(\log n)^{(d-1)/2}$ as it happens many times in this field. 
\end{remark}
From \eqref{wor-gel} and Theorem \ref{main} we have the estimate for the worst-case error of the class $\Lambda^{\rm all}$.
\begin{corollary}
Under the conditions of Theorem \ref{main}, for the embedding ${\rm App}:\ S^t_{p_1,p_1}B(\Omega) \to L_{p_2}(\Omega)$  we have
\[
e^{\rm wor}(n, \Lambda^{\rm all})\asymp n^{-\alpha}(\log n)^{(d-1)\beta} , \qquad n \ge2\, , 
 \]
where $\alpha$ and $\beta$ are given in Theorem \ref{main}.
\end{corollary}

\subsubsection*{A comparison with approximation numbers}
Since $a_n=e^{\rm wor-lin}(n,\Lambda^{\rm all}) $ and $c_n \asymp e^{\rm wor}(n,\Lambda^{\rm all})$, it is reasonable to compare Gelfand and approximation numbers of the embedding ${\rm App}: S^t_{p_1,p_1}B(\Omega)\to L_{p_2}(\Omega)$. The asymptotic behaviour of approximation numbers is given in the following theorem. 
\begin{theorem} \label{known}
Let $1< p_1,p_2< \infty $ and $ t>(\frac{1}{p_1}-\frac{1}{p_2})_+$. Then we have
$$a_n({\rm App}: S^t_{p_1,p_1}B(\Omega)\to L_{p_2}(\Omega))\asymp n^{-\alpha}(\log n)^{(d-1)\beta}\, , \ \ \ n\geq 2,$$
where
\begin{enumerate}
\item {\makebox[5.5cm][l]{$\alpha=t $, $\beta=t+(\frac{1}{2}-\frac{1}{p_1})_+$} if\ \  $ p_2\leq p_1 $};
\item {\makebox[5.5cm][l]{$\alpha=\beta=t-\frac{1}{p_1}+\frac{1}{p_2}$} if\ \  $ p_1\leq p_2\leq 2$ or $2\leq p_1\leq p_2$};
\item {\makebox[5.5cm][l]{$\alpha=\beta=t-\frac{1}{p_1}+\frac{1}{2}$} if\ \  $   2< p_2 <p_1'  $,\ $t>\frac{1}{p_1}$};
\item {\makebox[5.5cm][l]{$\alpha=\beta=t-\frac{1}{2}+\frac{1}{p_2}$} if\ \  $   2\leq  p_1' <p_2 $,\ $t>1-\frac{1}{p_2}$ };
\item {\makebox[5.5cm][l]{$\alpha=\frac{p_1'}{2}(t-\frac{1}{p_1}+\frac{1}{p_2})$, $\beta=\frac{2\alpha}
{p_1'}$} if\ \ $2\leq  p_1' <p_2 $,\ $t<1-\frac{1}{p_2}$}.
\end{enumerate}
\end{theorem}
%&&&&&&&&&&&&&&&&&&&&&&&&&&&&&&&&&&&&&&&&&&&&&&&&&&&&&&&&&&&&&&&&&&&&&&&&&&&&&&&&&&&&
%&&&&&&&&&&&&&&&&&&&&&&&&&&&&&&&&&&&&&&&&&&&&&&&&&&&&&&&&&&&&&&&&&&&&&&&&&&&&&&&&&&&&
\begin{remark}\rm Parts (i)-(iv) have been proved by Romanyuk \cite{Rom1,Rom2} and Bazarkhanov \cite{Baz7}. Part (v) follows analogously to the Gelfand case, see Remark \ref{app} below. 
\end{remark}
 The difference of Gelfand and approximation numbers of the embedding ${\rm App}: S^t_{p_1,p_1}B(\Omega)\to L_{p_2}(\Omega)$ is illustrated in the following figure, see  Theorems \ref{main}, \ref{known}  and Proposition \ref{conj}. We assume the dimension $d\geq 2$.

$$
\begin{tikzpicture}
\fill (0,0) circle (1.5pt);
\draw[->, ](0,0) -- (7,0);
\draw[->, ] (0,0) -- (0,6.8);
\draw (3,0)-- (3,6);
\draw (-0.06,3) -- (0.06,3);
\draw (0,6) -- (6,6);
\draw (6,0) -- (6,6);
\draw (3,3) -- (6,0);
\node[below] at (0,0) {$0$};
\node [below] at (3,0) {$\frac{1}{2}$};
\node [left] at (0,3) {$\frac{1}{2}$};
\node [left] at (0,6) {$ 1$};
\node [left] at (0,6.6) {$\frac{1}{p_2}$};
\node [below] at (6,0) {$1$};
\draw (3, -0.05) -- (3, 0.05);
\node [below] at (7,0) {$\frac{1}{p_1}$};
\node [] at (4.5,3.8) {$\lim\limits_{n \to \infty}\, \frac{c_n}{a_n} = 0  $};
\node [] at (3.95,1.2) {\small L: $c_n\asymp a_n$};
\node [] at (1.5,3) {$c_n\asymp a_n$};
\node [] at (4.25,0.3) {\small H:$\lim\limits_{n \to \infty}\, \frac{c_n}{a_n} = 0 $};

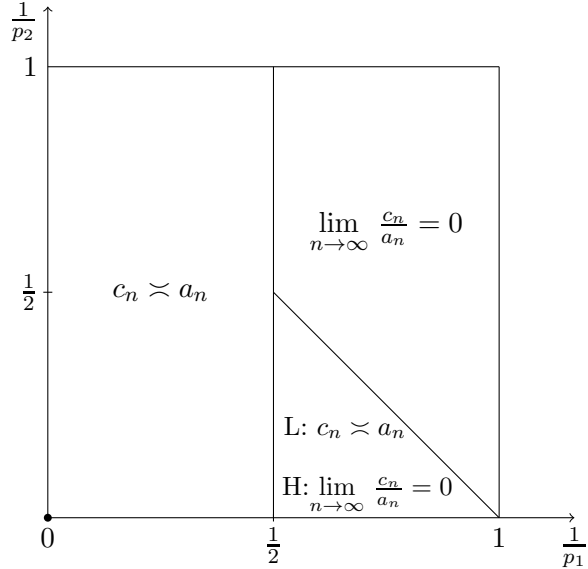
\node [right] at (-2.2,-1) {Figure 1. Comparison of Gelfand and approximation numbers};
\end{tikzpicture}
$$
Here H refers to the domain of ``high smoothness", i.e., $t>1-\frac{1}{p_2}$ and L refers to ``low smoothness", i.e., $t<1-\frac{1}{p_2}$. Figure 1 indicates that Gelfand numbers and approximation numbers show similar behaviour if either  $p_1\geq 2$ or  $   2\leq  p_1' <p_2 $,\ $t<1-\frac{1}{p_2}$, i.e., $c_n\asymp a_n$. This implies that in those cases, 
nonlinear algorithms for approximation problem  $
{\rm App}:  S^t_{p_1,p_1}B(\Omega) \to L_{p_2}(\Omega)
$ are not essentially better than linear algorithms. In other cases Gelfand number are essentially smaller than approximation numbers, i.e., $\lim_{n\to \infty}\frac{c_n}{a_n}=0$.

Now we proceed to the extreme cases. Since $L_{\infty}(\Omega)$ has the metric extension property, see \cite[Proposition C.3.2.2]{Pi-80} and also \cite[page 36]{Pin85}, we have $c_n(T)=a_n(T)$ for all linear bounded operator $T$ from Banach spaces $X$ into $L_{\infty}(\Omega)$, see \cite[Proposition 11.5.3]{Pi-80}. From this we can extend the result in Theorem \ref{ex} (ii) for approximation numbers.
\begin{theorem}\label{ex-2}
{\rm (i)} Let $2\leq p<\infty$ and $t>0$. Then we have
 \[
 a_n({\rm App}:\ S^t_{p,p}B(\Omega) \to L_{1}(\Omega))\asymp n^{-t}(\log n)^{(d-1)( t-\frac{1}{p}+\frac{1}{2})}, \qquad n \ge2\, .
  \]
{\rm (ii)} Let $1<p\leq 2$ and  $t>1$. Then we have
 \[
 a_n({\rm App}:\ S^t_{p,p}B(\Omega) \to L_{\infty}(\Omega))\asymp n^{-t+\frac{1}{2}}(\log n)^{(d-1)(t-\frac{1}{p}+\frac{1}{2})}, \qquad n \ge2\, .
  \]
\end{theorem}
\begin{remark}\rm (i) The proof of part (i) can be found in \cite{Rom9}. Theorems \ref{ex} and \ref{ex-2} indicate that  if $2\leq p<\infty$ and $t>0$ then 
\beqq
c_n({\rm App}:\ S^t_{p,p}B(\Omega) \to L_{1}(\Omega)) \asymp a_n({\rm App}:\ S^t_{p,p}B(\Omega) \to L_{1}(\Omega)).
\eeqq
(ii) We wish to mention that the study of approximation of functions with mixed smoothness in the uniform norm ($L_{\infty}$-norm) is more difficult. Beside the above result, there is only a small number of cases, where the exact order of $a_n({\rm App} : S^t_{p,q}B(\Omega) \to L_{\infty}(\Omega))$ (in this case $a_n=c_n$), if $n$ tends to infinity, has been found. We refer to comments and open problems presented in the survey \cite[Sections 4.5 and 4.6]{DTU-16}.
\end{remark}
\subsubsection*{Gelfand numbers of embeddings of Sobolev spaces with dominating mixed smoothness}

For better understanding and completeness we shall give the asymptotic behaviour of Gelfand numbers of embeddings of Sobolev spaces with dominating mixed smoothness. Let $1<p<\infty$ and $t\in \re$. Then $S^t_pH(\Omega)$  denotes the Sobolev spaces of fractional order with dominating mixed smoothness. These spaces represent special cases of the Triebel-Lizorkin spaces with dominating mixed smoothness $S^t_{p,q}F(\Omega)$,  i.e., $S^{t}_{p,2}F(\Omega) = S^t_pH(\Omega)$ in the sense of equivalent norms, see Section \ref{besov}. In the case $p=2$ we have $S^t_2H(\Omega)=S^t_{2,2}B(\Omega)=H^t_{\mix}(\Omega)$, see Section \ref{intro}. It is well-known that for $t=m\in \N$ 
\beqq
S^m_pH(\R):= \Big\{f\in L_p(\R): \|f|S^m_pH(\R)\|:=\sum_{|\alpha|_{\infty} \leq m}\|D^{\alpha}f|L_p(\R)\|<\infty\Big\}\,  
\eeqq 
in the sense of equivalent norms. Here $\alpha=(\alpha_1,...,\alpha_d)\in \N_0^d$ and $|\alpha|_{\infty}=\max_{i=1,...,d} |\alpha_i|$. Using  \eqref{dual}, lifting properties of Sobolev spaces with dominating mixed smoothness, see \cite[Section 2.2.6]{ST} and $ (S^t_{p}H(\Omega))'=S^{-t}_{p'}H(\Omega)$, see \cite[Section 5.5]{Hansen}, we obtain
\beqq
\begin{split}
c_n({\rm App}: S^t_{p_1}H(\Omega)\to L_{p_2}(\Omega))& = d_n({\rm App}: L_{p_2'}(\Omega) \to S^{-t}_{p_1'}H(\Omega)) \\ &\asymp d_n({\rm App}: S^t_{p_2'}H(\Omega)\to L_{p_1'}(\Omega)).
\end{split}
\eeqq 
Here $1<p_1,p_2<\infty$, $t>(\frac{1}{p_1}-\frac{1}{p_2})_+$ and $p_1',p_2'$ are conjugates of $p_1,p_2$ respectively. The behaviour of Kolmogorov numbers in such a context has been investigated at several places \cite{Rom6,Rom7,Rom4,Rom8,Rom5} and \cite{Baz7,Ga1,Tem}. Using the result on Kolmogorov numbers in the already mentioned references we obtain the following theorem, see also \cite[Section 9.7]{DTU-16}.
\begin{theorem}\label{gel-so}
Let $1< p_1,p_2< \infty $ and $ t>\big(\frac{1}{p_1}-\frac{1}{p_2}\big)_+$. Then we have
$$c_n({\rm App}: S^t_{p_1}H(\Omega)\to L_{p_2}(\Omega))\asymp  n^{-\alpha}(\log n)^{(d-1)\alpha}\, ,\ \ \ n\geq 2,$$
where
\begin{enumerate}
\item {\makebox[4.5cm][l]{$\alpha=t-\big(\frac{1}{p_1}-\frac{1}{p_2}\big)_+$} if\ \  $p_2\leq p_1 $ or $2\leq p_1\leq p_2 $};
\item {\makebox[4.5cm][l]{$\alpha=t-(\frac{1}{2}-\frac{1}{p_2})_+$} if\ \ $p_1\leq \min(p_2,2)$, $t>\max(\frac{1}{2},1-\frac{1}{p_2})$}.
\end{enumerate}
\end{theorem}
For the comparison of Gelfand numbers to approximation numbers of the embedding ${\rm App}: S^t_{p_1}H(\Omega)\to L_{p_2}(\Omega)$ with $1<p_1,p_2<\infty$ we refer to \cite[Section 9.7]{DTU-16}. Similar as in proof of Theorem \ref{ex} we obtain the behaviour of the Gelfand numbers in the extreme situations.
\begin{theorem}\label{ex-3}
{\rm (i)} Let $1< p<\infty$ and $t>0$. Then we have
 \[
 c_n({\rm App}:\ S^t_{p}H(\Omega) \to L_{1}(\Omega))\asymp  a_n({\rm App}:\ S^t_{p}H(\Omega) \to L_{1}(\Omega))\asymp n^{-t}(\log n)^{(d-1)t}, \qquad n \ge2\, .
  \]
{\rm (ii)} Let $1<p\leq 2$ and  $t>1$. Then we have
 \[
 c_n({\rm App}:\ S^t_{p}H(\Omega) \to L_{\infty}(\Omega))= a_n({\rm App}:\ S^t_{p}H(\Omega) \to L_{\infty}(\Omega))\asymp n^{-t+\frac{1}{2}}(\log n)^{(d-1)t}, \quad n \ge2\, .
  \]
\end{theorem}
\begin{remark}
\rm (i) The asymptotic behaviour of approximation numbers in part (i) has been proved by Romanyuk \cite{Rom9}.\\
(ii) Recall that part (ii) in Theorem \ref{ex-3} still holds true if $p=2$ and $t>\frac{1}{2}$ since $S^t_2H(\Omega)=S^t_{2,2}B(\Omega)$ in the sense of equivalent norms, see Proposition \ref{known1} and Remark \ref{remark}.\\
(iii) Observe, to prove Theorem \ref{main} by applying the same duality argument as in Theorem \ref{gel-so} we would need to know
\beqq
d_n({\rm App}: S^t_{p_2'}H(\Omega)\to S^0_{p_1',p_1'}B(\Omega)),
\eeqq
since $(S^{t}_{p_1,p_1}B(\Omega))'=S^{-t}_{p_1',p_1'}B(\Omega)$, see \cite[Section 2.3.8]{Hansen}. However, these numbers are not investigated except in a few special cases, e.g., $p_1=2$. 
\end{remark}
\section{Besov and Triebel-Lizorkin spaces with dominating mixed smoothness  }
\label{besov}

\subsection{Spaces on $\R$ and on the unit cube}
\label{space2}
Let us first introduce the Besov space with dominating mixed smoothness $ S^{t}_{p,p}B(\re^d)$. Detailed treatments of these spaces are given at various places, we refer to the monographs \cite{Am,ST}, see also \cite{Baz1,Baz2,Baz3} and \cite{Vy-06}. In this section we shall review the spaces $ S^{t}_{p,p}B(\re^d)$ by using the Fourier analytic
approach. 

Let $\cs (\R)$ be the Schwartz space of all complex-valued rapidly decreasing infinitely differentiable  functions on $\R$. 
The topological dual, the class of tempered distributions, is denoted by $\cs'(\R)$ (equipped with the weak topology). We denote the Fourier transform and its inverse on $\cs(\R)$ by $\cf$ and $\cfi$. Both $\cf$ and $\cfi$ are extended to $\cs'(\R)$ in the usual way. 
Let   $\varphi_0(\xi)\in C_0^{\infty}({\re})$ with $\varphi_0(\xi) = 1$ on $[-1,1]$ and $\supp\varphi_0 \subset [-\frac{3}{2},\frac{3}{2}]$. For $j\in \N$ we define
      $$
         \varphi_j(\xi) = \varphi_0(2^{-j}\xi)-\varphi_0(2^{-j+1}\xi)\,, \ \ \ \xi\in \re.
      $$
   For $\bar{k} = (k_1,...,k_d) \in {\N}_0^d$ the function
    $\varphi_{\bar{k}}(x) \in C_0^{\infty}(\R)$ is defined as
    $$
        \varphi_{\bar{k}}(x) := \varphi_{k_1}(x_1)\cdot...\cdot
         \varphi_{k_d}(x_d)\,,\quad x\in \R.
    $$
   
%&&&&&&&&&&&&&&&&&&&&&&&&&&&&&&&&&&&&&&&&&&&&&&&&&&&&&&&&&&&&&&&&&&&&&&&
%&&&&&&&&&&&&&&&&&&&&&&&&&&&&&&&&&&&&&&&&&&&&&&&&&&&&&&&&&&&&&&&&&&&&&&&

\begin{definition}\label{def-be} Let $1<p< \infty$ and $t \in\re$. The Besov space with dominating mixed smoothness $ S^{t}_{p,p}B(\re^d)$ is the
collection of all tempered distributions $f \in \mathcal{S}'(\R)$
such that
\[
 \|\, f \, |S^{t}_{p,p}B(\R)\| :=
\Big(\sum\limits_{\bar{k}\in \N_0^d} 2^{t|\bar{k}|_1  p}\, \|\, \cfi[\varphi_{\bar{k}}\, \cf f](\, \cdot \, )
|L_p(\re^d)\|^p\Big)^{1/p} <\infty.
\] 
\end{definition}
\begin{remark}\rm 
(i) If $d=1$ we obtain $
 S^{t}_{p,p} B(\re) = B^t_{p,p}(\re)  $ where $B^t_{p,p}(\re)$ is the isotropic Besov space on $\re$.  There is an extensive literature about isotropic Besov spaces, we refer to the monographs of Nikol'skij \cite{Ni} and Triebel \cite{Tr83,Tr92,Tr06}. Probably, one of the most interesting properties of Besov spaces with dominating mixed smoothness consists in the cross-norm, i.e., if $f_i \in B^t_{p,p}(\re)$ for $ i=1, ... , d$ then
\[
 f(x) = \prod_{i=1}^d f_i (x_i)  \in S^t_{p,p}B(\R) \qquad\text{and}\qquad  
 \| \, f \, | S^t_{p,p}B(\R)\| = \prod_{i=1}^d \|\, f_i \, |B^t_{p,p} (\re)\| \, .
\]
(ii) If $1<p<\infty$ and $t>0$, then the scale $S^t_{p,p}B(\R)$ can be characterized by differences, see \cite[Chapter 2]{ST}, but see also \cite{Am,U1}. 
\end{remark}

The space $S^{t}_{p,p}B(\re^{d})$ is actually a $d$-fold tensor product of the space $B^{t}_{p,p}(\re)$. Tensor products of Besov spaces have been investigated in \cite{SUt,SUspline}. For $1<p<\infty$, let $\sigma_p$ denote the $p$-nuclear tensor norm. Concerning the basic notions of tensor products of Banach spaces and basic properties of the  $p$-nuclear tensor norm we refer to \cite{LiCh}, but see also \cite{DF}. We have the following result.
\begin{proposition}\label{tensor-r} Let $d> 1$, $t \in \re$ and $1<p<\infty$. Then the following formula
\beqq
 S^{t}_{p,p}B(\re^{d}) = B^{t}_{p,p}(\re)\otimes_{\sigma_p} \cdots \otimes_{\sigma_p}
B^{t}_{p,p}(\re) \ \ \ \text{($d$ times)}
\eeqq
holds true in the sense of equivalent norms.
\end{proposition}
\begin{remark}
\rm The proof of Proposition  \ref{tensor-r} can be found in \cite{SUt}. Tensor product of more than two spaces should be understood as iterated tensor products, i.e., $X\otimes_{\sigma_p}Y\otimes_{\sigma_p}Z = X\otimes_{\sigma_p}(Y\otimes_{\sigma_p}Z )$.
\end{remark}
For later use, let us recall the lifting properties of Besov spaces with dominating mixed smoothness, see \cite[Section 2.2.6]{ST}.
\begin{theorem}\label{lift}
Let $t,r\in \re$ and $1<p<\infty$.  We define the lifting operator by
\beqq
I_{r}f :=\cf^{-1}[(1+ \xi_1^2)^{r/2}\cdots (1+\xi_d^2)^{r/2}\cf f]\,,\qquad f\in \cs' (\R),\ \ \xi=(\xi_1,...,\xi_d)\in \R.
\eeqq
Then $I_{r}$ maps $S^{t}_{p,p}B(\R)$ isomorphically onto $S^{t-r}_{
 p,p} B(\R)$ and  $\| I_{r}f|S^{t-r}_{ p,p} B(\R)\|$ is an equivalent norm in $S^{t}_{ p,p}B(\R)$.
\end{theorem}

We proceed by introducing  Triebel-Lizorkin spaces with dominating mixed smoothness $S^t_{p,q}F(\R)$ which will be useful in our proofs of the main results. We refer  to \cite[Chapter 2]{ST}.
\begin{definition}\label{def-tr} Let $1 < p,q< \infty$. The Triebel-Lizorkin space with dominating mixed smoothness $ S^{t}_{p,q}F(\re^d)$ is the
collection of all tempered distributions $f \in \mathcal{S}'(\R)$
such that
\[
 \|\, f \, |S^{t}_{p,q}F(\R)\| :=
\Big\| \Big(\sum\limits_{\bar{k}\in \N_0^d} 2^{t|\bar{k}|_1  q}\, |\, \cfi[\varphi_{\bar{k}}\, \cf f](\, \cdot \, )|^q \Big)^{1/q} \Big|L_p(\re^d)\Big\| <\infty.
\]
\end{definition}
\begin{remark}\rm  (i) In view of Definitions \ref{def-be} and \ref{def-tr} we have  $S^t_{p,p}B(\R)=S^t_{p,p}F(\R)$.\\
(ii) If $d=1$ we obtain $
 S^{t}_{p,q} F(\re) = F^t_{p,q}(\re)   $. The classes   $F^t_{p,q}(\R)$ are the isotropic Triebel-Lizorkin spaces, we refer again to the monographs  \cite{Ni,Tr83,Tr92,Tr06}. Triebel-Lizorkin spaces with dominating mixed smoothness  have a cross-norm, i.e., if $f_i \in F^t_{p,q}(\re)$ for $ i=1, ... , d$, then we have
\[
 f(x) = \prod_{i=1}^d f_i (x_i)  \in S^t_{p,q}F(\R) \qquad\text{and}\qquad  
 \| \, f \, | S^t_{p,q}F(\R)\| = \prod_{i=1}^d \|\, f_i \, |F^t_{p,q} (\re)\| \, .
\]
(iii) Sobolev spaces with dominating mixed smoothness $S^t_pH(\R)$ represent special cases of Triebel-Lizorkin classes, i.e., $ S^{t}_{p,2}F(\R) = S^t_pH(\R)$ ($1<p<\infty$) in the sense of equivalent norms, see \cite[Theorem 2.3.1]{ST}. In case $t=0$ we get back the Littlewood-Paley assertion $L_p(\R)=S^0_{p,2}F(\R) $, see Nikol'skij \cite[1.5.6]{Ni}.
\end{remark}
Since the spaces $S^t_{p,p}B(\R)$ and $L_p(\R)$ are special cases of the classes $S^t_{p,q}F(\R)  $, from now on we will work with the scale $S^t_{p,q}F(\R)$. We now turn to the spaces on unit cube $\Omega$. For us it will be convenient to define spaces on $\Omega$ by restrictions. By $D'(\Omega)$ we denote the set of all complex-valued distributions on $\Omega$.
\begin{definition} \label{defomega}
 Let $ 1<p,q< \infty$ and $t\in \re$. Then
   $S^{t}_{p,q}F(\Omega)$ is the space of all $f\in D'(\Omega)$ such that there exists a distribution $g\in
   S^{t}_{p,q}F(\R)$ satisfying $f = g|_{\Omega}$. It is endowed with the quotient norm
   $$
      \|\, f \, |S^{t}_{p,q}F(\Omega)\| = \inf \Big\{ \|g|S^{t}_{p,q}F(\R)\|~: \quad g|_{\Omega} =
      f \Big\}\,.
   $$   
\end{definition}
\begin{remark}\rm Of course, we have $S^t_{p,p}B(\Omega)=S^t_{p,p}F(\Omega)$. For the existence of a linear extension operator from $S^t_{p,p}B(\Omega)$ into $S^t_{p,p}B(\R)$ and the intrinsic characterizations (by differences) of the spaces $S^t_{p,p}B(\Omega)$ we refer to \cite[Section 1.2.8]{Tr10}.
\end{remark}
%&&&&&&&&&&&&&&&&&&&&&&&&&&&&&&&&&&&&&&&&&&&&&&&&&&&&&&&&&&&&&&&&&&&&&&&&&&&
%&&&&&&&&&&&&&&&&&&&&&&&&&&&&&&&&&&&&&&&&&&&&&&&&&&&&&&&&&&&&&&&&&&&&&&&&&&&

\subsection{Sequence spaces related to function spaces with dominating mixed smoothness}
%&&&&&&&&&&&&&&&&&&&&&&&&&&&&&&&&&&&&&&&&&&&&&&&&&&&&&&&&&&&&&&&&&&&&&&&&&&&&&&&&&&&&&&&&&&&&&&&&&&&&&&&&&&&&&&&&&&&&&&
%&&&&&&&&&&&&&&&&&&&&&&&&&&&&&&&&&&&&&&&&&&&&&&&&&&&&&&&&&&&&&&&&&&&&&&&&&&&&&&&&&&&&&&&&&&&&&&&&&&&&&&&&&&&&&&&&&&&&&&

We first recall wavelet bases of Triebel$-$Lizorkin spaces with dominating mixed smoothness.
Let $N \in \N$. Then there exists $\psi_0, \psi_1 \in C^N(\re) $, compactly supported, 
\[
\int_{-\infty}^\infty t^m \, \psi_1 (t)\, dt =0\, , \qquad m=0,1,\ldots \, , N\, , 
\]
such that
$\{ 2^{j/2}\, \psi_{j,m}:\ j \in \N_0, \: m \in \zz\}$, where
\[
\psi_{j,m} (t):= \left\{ \begin{array}{lll}
\psi_0 (t-m) & \qquad & \mbox{if}\quad j=0, \: m \in \zz\, , 
\\
\sqrt{1/2}\, \psi_1 (2^{j-1}t-m) & \qquad & \mbox{if}\quad j\in \N\, , \: m \in \zz\, , 
            \end{array} \right.
\]
is an orthonormal basis in $L_2 (\re)$, see \cite{woj}. 
Consequently, the system
\[
\Psi_{\bar{\nu}, \bar{m}} (x) := \prod_{\ell=1}^d \psi_{\nu_\ell, m_\ell} (x_\ell)\, \qquad \bar{\nu} \in \N_0^d, \, \bar{m} \in \Z\, , 
\]
is a tensor product wavelet basis of $L_2 (\R)$. Vybiral \cite[Theorem 2.12]{Vy-06} has proved the following.

\begin{lemma}\label{wavelet}
Let $1< p,q <\infty$ and $t\in \re$. There exists $N=N(t,p,q) \in \N$ such that the mapping 
\beqq
{\mathcal W}: \quad f \mapsto (2^{|\bar{\nu}|_1}\langle f, \Psi_{\bar{\nu},\bar{m}} \rangle)_{\bar{\nu} \in \N_0^d\, , \, \bar{m} \in \Z} 
\eeqq
is an isomorphism of $S^t_{p,q}F(\R)$ onto $s^t_{p,q}f$.
\end{lemma}

We put
\beqq
A_{\bar{\nu}}^{\Omega}:= 
\Big\{\bar{m}\in \mathbb{Z}^d:\ \supp \Psi_{\bar{\nu},\bar{m}} \cap\Omega \neq \emptyset\Big\}\, ,\qquad \bar{\nu}\in\mathbb{N}_0^d\, .
\eeqq
For given $ f \, \in S_{p,q}^t F(\Omega)$ let $\ce f$ be an element of $ S_{p,q}^t F(\R)$ s.t. 
\[
\| \, \ce f \, | S_{p,q}^t F(\R)\| \le 2 \, \| \, f \, | S_{p,q}^t F(\Omega)\|
\qquad \mbox{and} \qquad (\ce f)_{|_\Omega} = f \, .
\]
We define
\[
g:= \sum_{\bar{\nu} \in \N_0^d} \sum_{\bar{m} \in A_{\bar{\nu}}^{\Omega}} 2^{|\bar{\nu}|_1} \, \langle \ce f, \Psi_{\bar{\nu},\bar{m}} \rangle \, \Psi_{\bar{\nu}, \bar{m}}\, .
\]
Then it follows that $g \in S_{p,q}^t F(\R)$, $g_{|_\Omega} = f $, 
\[
\supp g \subset \{x \in \R: ~ \max_{j=1, \ldots \, , d} |x_j| \le c_1\} \quad \mbox{and} \qquad 
\| \, g \, | S_{p,q}^t F(\R)\| \le c_2 \, \| \, f \, | S_{p,q}^t F(\Omega)\|\, .
\]
Here $c_1,c_2$ are independent of $f$. For this reason we define the following sequence spaces

\begin{definition}
Let $1<p,q<\infty$ and $t\in \mathbb{R}$.\\
{\rm (i)} If 
$$
\lambda=\lbrace \lambda_{\bar{\nu},\bar{m}}\in\mathbb{C}:\bar{\nu}\in \mathbb{N}_0^d,\ \bar{m}\in A_{\bar{\nu}}^{\Omega} \rbrace \, ,
$$
then we define
\[
s_{p,q}^{t,\Omega}f :=\Big\lbrace\lambda: \| \lambda|s_{p,q}^{t,\Omega}f\| =\Big\|\Big(\sum_{\bar{\nu}\in \mathbb{N}_0^d}\sum_{\bar{m}\in A_{\bar{\nu}}^{\Omega}}|2^{|\bar{\nu}|_1 t}\lambda_{\bar{\nu},\bar{m}}
\chi_{\bar{\nu},\bar{m}}(\cdot) |^q\Big)^{\frac{1}{q}}\Big|L_p(\mathbb{R}^d)\Big\|<\infty \Big\rbrace\, .
\]
{\rm (ii)} If $\mu\in \N_0$ and 
$$
\lambda=\lbrace \lambda_{\bar{\nu},\bar{m}}\in\mathbb{C}:\bar{\nu}\in \mathbb{N}_0^d,\ |\bar{\nu}|_1=\mu,\ \bar{m}\in A_{\bar{\nu}}^{\Omega} \rbrace \, ,
$$
then we define
$$(s_{p,q}^{t,\Omega}f)_{\mu}=\Big\lbrace\lambda: \| \lambda|(s_{p,q}^{t,\Omega}f)_{\mu}\| =
\Big\|\Big(\sum_{|\bar{\nu}|_1=\mu}\sum_{\bar{m}\in A_{\bar{\nu}}^{\Omega}}|2^{|\bar{\nu}|_1t}\lambda_{\bar{\nu},\bar{m}}\chi_{\bar{\nu},\bar{m}}(\cdot) 
|^q\Big)^{\frac{1}{q}}\Big|L_p(\mathbb{R}^d)\Big\|<\infty \Big\rbrace \, .$$

\end{definition}

Later on we shall need the following lemmas, see \cite{Hansen,KiSi,Vy-06}.

\begin{lemma}\label{ba1}
{\rm (i)} Let $\bar{\nu}\in \mathbb{N}_0^d$ and $\mu\in \mathbb{N}_0$. Then we have
$$ \#(A_{\bar{\nu}}^{\Omega})\asymp 2^{|\bar{\nu}|_1}\qquad\text{and}\qquad D_{\mu}=\sum_{|\bar{\nu}|_1=\mu}\#(A_{\bar{\nu}}^{\Omega})\asymp \mu^{d-1}2^{\mu}. \ $$
The equivalence constants do not depend on $\mu\in \mathbb{N}_0$.\\
{\rm (ii)} Let $1<p< \infty$ and $t\in \mathbb{R}$. Then
$$(s_{p,p}^{t,\Omega}f)_{\mu}=2^{\mu(t-\frac{1}{p})}\ell_p^{D_{\mu}}\,,\qquad \mu\in \mathbb{N}_0\,.$$
\end{lemma}
\begin{lemma}\label{ba2}
{\rm (i)} Let $1<p_1,p_2,q_1,q_2<\infty$ and $t\in \mathbb{R}$. Then
\[
\| \, id^*_{\mu} \, : (s^{t,\Omega}_{p_1,q_1}f)_\mu \to (s^{0,\Omega}_{p_2,q_2}f)_\mu\|
\lesssim 2^{\mu\big(-t+(\frac{1}{p_1}-\frac{1}{p_2})_+\big)}\mu^{(d-1)(\frac{1}{q_2} - \frac{1}{q_1})_+}
\]
with a constant behind $\lesssim$ independent of $\mu\in \mathbb{N}_0$.\\
{\rm (ii)} Let $1<p_1<p_2<\infty$, $1<q_1,q_2< \infty$ and $t \in \re$. Then
$$ \|id^*_{\mu}: (s_{p_1,q_1}^{t,\Omega}f)_{\mu}\to (s_{p_2,q_2}^{0,\Omega}f)_{\mu} \| \lesssim 2^{\mu(-t+\frac{1}{p_1}-\frac{1}{p_2})}$$
with a constant behind $\lesssim$ independent of $\mu\in \mathbb{N}_0$.
\end{lemma}
%\section{Weyl and Bernstein numbers of embeddings $id^*_{\mu}: (s_{p_1,2}^{t,\Omega}f)_{\mu}\to (s_{p_2,2}^{0,\Omega}f)_{\mu}$}

%&&&&&&&&&&&&&&&&&&&&&&&&&&&&&&&&&&&&&&&&&&&&&&&&&&&&&&&&&&&&&&&&&&&&&&&&&&&&&&&&&&&&&&&&&&&&&&&&&&&&&&&&&&&&&&&&&&&&&&
%&&&&&&&&&&&&&&&&&&&&&&&&&&&&&&&&&&&&&&&&&&&&&&&&&&&&&&&&&&&&&&&&&&&&&&&&&&&&&&&&&&&&&&&&&&&&&&&&&&&&&&&&&&&&&&&&&&&&&&

%&&&&&&&&&&&&&&&&&&&&&&&&&&&&&&&&&&&&&&&&&&&&&&&&&&&&&&&&&&&&&&&&&&&&&&&&&&&&&&&&&&&&&&&&&&&&&&&&&&&&&&&&&&&&&&&&&&&&&&
%&&&&&&&&&&&&&&&&&&&&&&&&&&&&&&&&&&&&&&&&&&&&&&&&&&&&&&&&&&&&&&&&&&&&&&&&&&&&&&&&&&&&&&&&&&&&&&&&&&&&&&&&&&&&&&&&&&&&&&

\section{Proofs}
\label{proof}

The proof is in some sense standard.  
By means of wavelet characterizations of Triebel$-$Lizorkin spaces we switch from the consideration of Gelfand numbers of  the embedding ${\rm App}: S^t_{p_1,p_1}B(\Omega) \to L_{p_2}(\Omega)$
to the Gelfand numbers of $id^*: s^{t, \Omega}_{p_1,p_1} f \to s^{0,\Omega}_{p_2,2}f$. 
Next,  we reduce the problem to the estimate of $c_{n}(id_{\mu}^*: (s^{t, \Omega}_{p_1,p_1} f)_{\mu} \to (s^{0,\Omega}_{p_2,2}f)_{\mu})$. In a further reduction step estimates of $c_{n_{\mu}}(id_{\mu}^*)$ are traced back to estimates of 
 $c_{n}(id_{p_1,p_2}^{D_\mu})$, see \eqref{idlp} for this notion. All what is needed about these numbers is collected in the following lemma.
 \begin{lemma}\label{gel}
 Let $1\leq p_1,p_2\leq \infty$ and $1\leq n\leq m<\infty$. Then if $p_1>1 $ we have
 \begin{subnumcases}{c_n(id_{p_1,p_2}^{m})\asymp}
 (m-n+1)^{\frac{1}{p_2}-\frac{1}{p_1}}&\text{if } $1\leq p_2\leq  p_1  $, \label{wth1}  \\
 \Big(\min\{1,m^{1-\frac{1}{p_1}}n^{-\frac{1}{2}}\}\Big)^{\frac{1/p_1-1/p_2}{1/p_1-1/2} }&\text{if } $1< p_1<  p_2\leq 2$,\label{wth2}   \\
 \max\Big\{m^{\frac{1}{p_2}-\frac{1}{p_1}}, \big(\sqrt{1-n/m}\,\big)^{\frac{1/p_1-1/p_2}{1/2-1/p_2}}\Big\}&\text{if } $2\leq p_1< p_2 $,\nonumber \\
  \max\Big\{m^{\frac{1}{p_2}-\frac{1}{p_1}},\min\{1,m^{1-\frac{1}{p_1}}n^{-\frac{1}{2}}\}  \sqrt{1-n/m}  \Big\}&\text{if } $1< p_1\leq 2<p_2$. \label{wth4}
 \end{subnumcases}
 \end{lemma}
For the proof of Lemma \ref{gel} we refer to Gluskin \cite{Glus}. The heart of the matter consists in the following assertion.
\begin{lemma}\label{weyl}
Let $1 < p_1 ,p_2< \infty$ and $t\in \re$. Then
\[
 c_n\big({\rm App}: S_{p_1,p_1}^t B(\Omega)\to L_{p_2}(\Omega)\big) \asymp c_n( id^*: s_{p_1,p_1}^{t,\Omega}f \to s_{p_2,2}^{0,\Omega}f)
\]
holds for all $n \in \N$.
\end{lemma}

\begin{proof} From the Littewood-Paley assertion $S_{p_2,2}^{0} F (\Omega) = L_{p_2} (\Omega)$ ($1<p_2<\infty$), see \cite[1.5.6]{Ni}, we have
\[
c_n\big({\rm App}: S_{p_1,p_1}^t B(\Omega)\to L_{p_2}(\Omega)\big) \asymp c_n\big({\rm App}: S_{p_1,p_1}^t B(\Omega)\to S_{p_2,2}^{0} F (\Omega)\big).
\]
Lemma \ref{wavelet} and property (c) of the $s$-numbers yield
$$c_n\big({\rm App}: S_{p_1,p_1}^t B(\Omega)\to S_{p_2,2}^{0} F (\Omega)\big) \asymp c_n( id^*: s_{p_1,p_1}^{t,\Omega}f \to s_{p_2,2}^{0,\Omega}f),$$
see also a related proof for Weyl numbers in \cite[Lemma 7.1]{KiSi}. From this the claim follows.
\end{proof}
\subsection{Gelfand numbers of embeddings of sequence spaces}
As the consequence of Lemma \ref{weyl}, in the following we shall deal with the behaviour of Gelfand numbers of the identity mapping
\[ 
id^*: \quad s_{p_1,p_1}^{t,\Omega}f\to s_{p_2,2}^{0,\Omega}f \, .
\]
To get a lower bound we use the following lemma.
\begin{lemma}\label{low}
For all $\mu \in \N_0$ and all $n \in \N$ we have 
\begin{equation}\label{eqlow1}
c_{n}(id_{\mu}^*: (s^{t, \Omega}_{p_1,p_1} f)_{\mu} \to (s^{0,\Omega}_{p_2,2}f)_{\mu})\leq c_n(id^*) \, .
\end{equation}
\end{lemma}

Lemma \ref{low} was proved for Weyl numbers in \cite[Lemma 6.10]{KiSi}. However we can follow the proof there and obtain the similar result for Gelfand numbers. Concerning the estimate from above the main idea is using the decomposition method, see \cite{Vy-06} and also \cite{KiSi}. We define the operators
$$ id_{\mu}  :\ s_{p_1,p_1}^{t,\Omega}f\to s_{p_2,2}^{0,\Omega}f\, , $$
where 
$$
(id_{\mu} \lambda)_{\bar{\nu},\bar{m}} := \begin{cases}
\lambda_{\bar{\nu},\bar{m}}&\text{if }\ |\bar{\nu}|_1=\mu,\\
0&\text{otherwise}.
\end{cases}
$$
We split $id^* :\ s_{p_1,p_1}^{t,\Omega}f\to s_{p_2,2}^{0,\Omega}f$
 into a sum of identities between building blocks
\beqq
id^* =   \sum_{\mu=0}^{J}\, id_{\mu}  + \sum_{\mu=J+1}^{L}\, id_{\mu}  + \sum_{\mu=L+1}^{\infty}\, id_{\mu} ,
\eeqq
where $ J$ and $L$ are at our disposal. These numbers $J$ and $L$ will be chosen in dependence on the parameters. The additivity and the monotonicity of the Gelfand numbers yield
\be\label{ws-16}
c_n(id^*)\leq \sum_{\mu=0}^{J}c_{n_{\mu}}(id_{\mu} )+\sum_{\mu=J+1}^{L}c_{n_{\mu}}(id_{\mu} )+\sum_{\mu=L+1}^{\infty} \|id_{\mu} \|  , 
\ee
where $n-1 = \sum_{\mu=0}^{L}(n_{\mu}-1)$. We observe that for $n\in \N$ and $\mu\in \N_0$ we have
\be\label{equ}
c_n\big( id_{\mu} :\ s_{p_1,p_1}^{t,\Omega}f\to s_{p_2,2}^{0,\Omega}f\big) = c_n\big(id_{\mu}^* :\ (s_{p_1,p_1}^{t,\Omega}f)_{\mu}\to (s_{p_2,2}^{0,\Omega}f)_{\mu}\big),
\ee
in particular, $\|id_{\mu}\| =\| id_{\mu}^*\|$. We have
$$ \|id_{\mu}^* \|\lesssim 2^{-\mu\big(t-(\frac{1}{p_1}-\frac{1}{p_2})_+\big)}\, \mu^{(d-1)(\frac{1}{2}-\frac{1}{p_1})_+}, $$
see Lemma \ref{ba2}, which results in the estimate
\be\label{ws-17}
\sum_{\mu=L+1}^{\infty} \|id_{\mu}^*\| \lesssim 2^{-L\big(t-(\frac{1}{p_1}-\frac{1}{p_2})_+\big) }L^{(d-1) (\frac{1}{2}-\frac{1}{p_1})_+} \, .
\ee
Now we choose $
n_{\mu} := D_{\mu}+1,\ \mu=0,1,....,J \, .
$
Then we get
\be\label{ws-23}
\sum_{\mu=0}^{J}n_{\mu} \, \asymp\, \sum_{\mu=0}^{J}\mu^{(d-1)}2^{\mu} \, \asymp\, J^{d-1}2^J
\ee
and $c_{n_{\mu}}(id_{\mu}^*)=0 $, see property (d) of $s$-numbers, 
which implies 
\be\label{ws-19}
\sum_{\mu=0}^{J}c_{n_{\mu}}(id_{\mu}^*)=0 \, .
\ee
Summarizing \eqref{ws-16}-\eqref{ws-17} and \eqref{ws-19} we have found
\begin{equation}\label{sum1}
c_n(id^*)\lesssim \sum_{\mu=J+1}^{L} \, c_{n_{\mu}}(id_{\mu}^*)+ 2^{-L\big(t-(\frac{1}{p_1}-\frac{1}{p_2})_+\big)}L^{(d-1) (\frac{1}{2}-\frac{1}{p_1})_+}\, .
\end{equation}
Now we turn to the problem to reduce the estimates for  Gelfand numbers $c_{n_{\mu}}(id_{\mu}^*)$ to estimates for 
$c_{n}(id_{p_1,p_2}^m)$. The following results were proved for Weyl numbers in \cite[Propositions 6.7, 6.8, Lemma 6.9]{KiSi}, but they are also true for Gelfand numbers.

\begin{lemma}\label{wichtig1}
Let $1<p_1< \infty$ and $t\in \mathbb{R}$. Then we have the following assertions.
\begin{enumerate}
\item If $1< p_2\leq 2$, then 
\begin{equation}\label{case1}
\mu^{(d-1)(-\frac{1}{p_2}+\frac{1}{2})}2^{\mu(-t+\frac{1}{p_1}-\frac{1}{p_2})} \, c_n(id_{p_1,p_2}^{D_{\mu}} )
\lesssim c_n(id_{\mu}^*) \lesssim 2^{\mu(-t+\frac{1}{p_1}-\frac{1}{p_2})}\, c_n(id_{p_1,p_2}^{D_\mu}).
\end{equation}
\item If $2\leq p_2<\infty$, then 
\begin{equation}\label{case2}
2^{\mu(-t+\frac{1}{p_1}-\frac{1}{p_2})}\, c_n(id_{p_1,p_2}^{D_{\mu}} )\lesssim c_n(id_{\mu}^*) .
\end{equation}
\end{enumerate}
\end{lemma}

\begin{lemma}
Let $1<p_1,p_2<\infty$ and $0<\varepsilon<p_2$. Then
\begin{equation}\label{case5}
c_n(id_{\mu}^*)\leq 2^{\mu(-t+\frac{1}{p_1}-\frac{1}{p_2})} \, c_n(id_{p_1,p_2-\varepsilon}^{D_{\mu}})\, .
\end{equation}
\end{lemma}
In addition, we need the following lemma.
\begin{lemma} Let $1<p_1,p_2<\infty$ and $0<\varepsilon$. Then
\be \label{case6}
2^{\mu( -t+\frac{1}{p_1} -\frac{1}{p_2})}\, c_n(id_{p_1,p_2+\varepsilon}^{D_\mu}) \lesssim
c_n(id^*_{\mu}).
\ee 
\end{lemma}
\begin{proof}
 We consider the following diagram

\tikzset{node distance=4cm, auto}

\begin{center}
\begin{tikzpicture}
 \node (H) {$(s^{t,\Omega}_{p_1,p_1}f)_\mu $};
 \node (L) [right of =H] {$(s^{0,\Omega}_{p_2+\varepsilon,p_2+\varepsilon}f)_\mu $};
 \node (L2) [right of =H, below of =H, node distance = 2cm ] {$(s^{0,\Omega}_{p_2,2}f)_\mu$};
 \draw[->] (H) to node {$id^2$} (L);
 \draw[->] (H) to node [swap] {$id^*_\mu$} (L2);
 \draw[->] (L2) to node [swap]{$id^1$} (L);
 \end{tikzpicture}
\end{center}
%$$
%\begin{tikzcd}[column sep=small]
%(s^{t,\Omega}_{p_1,p_1}f)_\mu \arrow[dr, "id_\mu"] \arrow[rr,"id_{2}"]& 
%&(s^{0,\Omega}_{p_2,p_2}f)_\mu \\
%& (s^{0,\Omega}_{p_2,2}f)_\mu \arrow[ur,"id_1"]
%\end{tikzcd}
%$$
\noindent
and obtain 
\be\label{1}
 c_n(id^2)\leq \| \, id^1\, \| \, c_n(id_{\mu}^*) ,
 \ee
see property (c) of $s$-numbers. By Lemma \ref{ba2} we have 
\be\label{2}
 \| \, id^1\, \| \lesssim 2^{\mu(\frac{1}{p_2}-\frac{1}{p_2+\varepsilon})}.
 \ee 
From Lemma \ref{ba1} we derive
$$ 2^{\mu( -t+\frac{1}{p_1} -\frac{1}{p_2+\varepsilon})} \, c_n(id_{p_1,p_2+\varepsilon}^{D_\mu}) \lesssim c_n(id^2) \, . $$
Inserting this and \eqref{2} into \eqref{1} we obtain the claimed estimate.
\end{proof}
\begin{proposition}\label{mot} Let $1<p_2<\infty $, $\max(2,p_2)\leq  p_1<\infty$ and $t>0$. Then
\[ 
c_n(id^* )\asymp n^{-t}(\log n)^{(d-1)(t -\frac{1}{p_1}+\frac{1}{2})} \, , \qquad n\geq 2.
\]
\end{proposition}
\begin{proof}
The upper estimates is a direct consequence of the inequality $c_n\leq a_n$ and Theorem  \ref{known} (i), see also \cite{Ki}. Concerning the estimate from below we first consider the case $p_2\leq 2\leq p_1$. From \eqref{eqlow1} and \eqref{case1} we have
\be \label{p2<2<p1}
\mu^{(d-1)(-\frac{1}{p_2}+\frac{1}{2})}2^{\mu(-t+\frac{1}{p_1}-\frac{1}{p_2})} \, c_n(id_{p_1,p_2}^{D_{\mu}} )
\lesssim c_n(id^*).
\ee 
Now we choose $n=[D_{\mu}/2]$ ($[x]$ denotes the integer part of the real number $x$) then   
\beqq
c_n(id_{p_1,p_2}^{D_{\mu}}) \asymp D_{\mu}^{\frac{1}{p_2}-\frac{1}{p_1}}\asymp (2^{\mu}\mu^{d-1})^{\frac{1}{p_2}-\frac{1}{p_1}}
\eeqq 
see \eqref{wth1}. Putting this into \eqref{p2<2<p1} we arrive at
\beqq
c_n(id^*)\gtrsim 2^{-\mu t}\mu^{(d-1)(\frac{1}{2}-\frac{1}{p_1})}.
\eeqq
Since $2^{\mu}\asymp \frac{n}{(\log n)^{d-1}}$ we obtain 
\beqq
c_n(id^* )\gtrsim n^{-t}(\log n)^{(d-1)(t -\frac{1}{p_1}+\frac{1}{2})}
\eeqq
for $n\asymp \mu^{d-1}2^{\mu}$, $\mu\in \N_0$. By monotonicity of Gelfand numbers, we extend this result to all $n\geq 2$. We estimate the lower bound for the case $2\leq p_2 \leq p_1$ by considering the chain of embeddings
\[ 
\quad s_{p_1,p_1}^{t,\Omega}f\hookrightarrow s_{p_2,2}^{0,\Omega}f \hookrightarrow s_{2,2}^{0,\Omega}f \, .
\]
From property (c) of the $s$-numbers we obtain
\beqq
c_n(id: s_{p_1,p_1}^{t,\Omega}f \to s_{2,2}^{0,\Omega}f ) \leq c_n(id^* ) \cdot \|id: s_{p_2,2}^{0,\Omega}f \to s_{2,2}^{0,\Omega}f\| \lesssim c_n(id^* ).
\eeqq
This together with the above result implies the desired estimate. The proof is complete.
\end{proof}
\begin{proposition}\label{hai} Let $1<p_1<2\leq p_2$ and $t>1-\frac{1}{p_2}$. Then
\[ 
c_n(id^* )\asymp n^{-t+\frac{1}{2}-\frac{1}{p_2}}(\log n)^{(d-1)(t-\frac{1}{p_1}+\frac{1}{p_2})} \, , \qquad n\geq 2.
\]
\end{proposition}
\begin{proof} First, the lower estimate follows from the relation $x_n \leq c_n $, see \eqref{xca}. We refer to \cite{KiSi} for the behaviour of Weyl numbers of this embedding. Now we turn to estimate from above. If $2<p_2$ we choose $\varepsilon>0$ such that $2<p_2-\varepsilon$. From \eqref{sum1} and  \eqref{case5} we have
\begin{equation}\label{sum1-1}
c_n(id^*)\lesssim \sum_{\mu=J+1}^{L} 2^{\mu(-t+\frac{1}{p_1}-\frac{1}{p_2})} \, c_{\mu}(id_{p_1,p_2-\varepsilon}^{D_{\mu}})+ 2^{-L\big(t- \frac{1}{p_1}+\frac{1}{p_2} \big)}.
\end{equation}
In case $p_2=2$ we choose $\varepsilon=0$ and use \eqref{case1}, then the estimate \eqref{sum1-1} still holds true. Now we define $$n_{\mu}=\big[D_{\mu}2^{(J-\mu)\lambda}\big], \qquad \mu=J+1,...,L.$$ 
Here $\lambda>1$ will be chosen later on. This together with \eqref{ws-23} guarantees that 
\be \label{2^J}
n=\sum_{\mu=0}^{L}(n_{\mu}-1)+1\asymp J^{d-1}2^J.
\ee 
In a view of \eqref{wth2} and \eqref{wth4} we have
\beqq
\begin{split} 
2^{\mu(-t+\frac{1}{p_1}-\frac{1}{p_2})} \, c_{\mu}(id_{p_1,p_2-\varepsilon}^{D_{\mu}}) 
& \lesssim 2^{\mu(-t+\frac{1}{p_1}-\frac{1}{p_2})} D_{\mu}^{1-\frac{1}{p_1}}[D_{\mu}2^{(J-\mu)\lambda}]^{-\frac{1}{2}}\\
& \lesssim 2^{\mu(-t+\frac{1}{2}-\frac{1}{p_2}+\frac{\lambda}{2})} \mu^{(d-1)(\frac{1}{2}-\frac{1}{p_1})} 2^{-\frac{J}{2}}.\\
\end{split}
\eeqq
Since $t>1-\frac{1}{p_2}$ we can choose $\lambda>1$ such that $-t+\frac{1}{2}-\frac{1}{p_2}+\frac{\lambda}{2}<0$. Consequently
\beqq
\begin{split}
 \sum_{\mu=J+1}^{L} 2^{\mu(-t+\frac{1}{p_1}-\frac{1}{p_2})} \, c_{\mu}(id_{p_1,p_2-\varepsilon}^{D_{\mu}}) 
 &\lesssim 2^{J(-t+\frac{1}{2}-\frac{1}{p_2}+\frac{\lambda}{2})} J^{(d-1)(\frac{1}{2}-\frac{1}{p_1})} 2^{-\frac{J}{2}}\\
&\lesssim 2^{J(-t+\frac{1}{2}-\frac{1}{p_2} )} J^{(d-1)(\frac{1}{2}-\frac{1}{p_1})}   .
\end{split}
\eeqq
Now choosing $L$ in \eqref{sum1-1} large enough and using $n\asymp 2^JJ^{d-1}$ we obtain
\beqq
c_{c2^JJ^{d-1}}(id^*)\lesssim 2^{J(-t+\frac{1}{2}-\frac{1}{p_2} )} J^{(d-1)(\frac{1}{2}-\frac{1}{p_1})}.
\eeqq
Substituting $n=c2^JJ^{d-1}$ in this inequality and using monotonicity arguments we get the estimate from above.
\end{proof}
\begin{proposition}\label{ba} Let $1<p_1\leq 2<p_2$ and $\frac{1}{p_1}-\frac{1}{p_2}<t<1-\frac{1}{p_2}$. Then
\[ 
c_n(id^* )\asymp n^{-\frac{p_1'}{2}(t-\frac{1}{p_1}+\frac{1}{p_2})}(\log n)^{(d-1)(t-\frac{1}{p_1}+\frac{1}{p_2})} \, , \qquad n\geq 2.
\]
\end{proposition}
\begin{proof}{\it Step 1.} Estimate from below. Since $p_2\geq 2$ we employ \eqref{eqlow1} and \eqref{case2}  to obtain
\beqq
2^{\mu( -t+\frac{1}{p_1} -\frac{1}{p_2})}\, c_n(id_{p_1,p_2}^{D_\mu}) \lesssim
c_n(id^*).
\eeqq
By choosing $n=[D_{\mu}^{2/p_1'}]$ we have from \eqref{wth4}
\beqq
2^{\mu( -t+\frac{1}{p_1} -\frac{1}{p_2})}  \lesssim
c_n(id^*).
\eeqq
Because of $2^{\mu}\asymp n^{\frac{p_1'}{2}}/(\log n)^{d-1}$ we arrive at 
\[ 
c_n(id^* )\gtrsim  n^{-\frac{p_1'}{2}(t-\frac{1}{p_1}+\frac{1}{p_2})}(\log n)^{(d-1)(t-\frac{1}{p_1}+\frac{1}{p_2})} \, .
\]
Again the monotonicity of Gelfand numbers implies the estimate for all $n\geq 2$.\\
{\it Step 2.} Estimate from above.  
We use the inequality \eqref{sum1-1}. Next we define
\be \label{n_mu}
 n_{\mu}:= \big[D_{\mu}\, 2^{\{(\mu-L)\beta+J-\mu \}}\big]\leq D_{\mu}/2\, , \qquad J+1 \le \mu \le L\, ,
\ee
 where $\beta >0$ will be fixed later on. This guarantees \eqref{2^J}. From \eqref{wth4} we obtain
 \beqq
 \begin{split} 
 2^{\mu(-t+\frac{1}{p_1}-\frac{1}{p_2})}c_n(id_{p_1,p_2-\varepsilon}^{D_{\mu}}) &\lesssim 2^{\mu(-t+\frac{1}{p_1}-\frac{1}{p_2})}  D_{\mu}^{\frac{1}{p_1'}} \big[D_{\mu}\, 2^{\{(\mu-L)\beta+J-\mu \}}\big]^{-\frac{1}{2}} \\
 & = 2^{\mu(-t+1-\frac{1}{p_2}-\frac{\beta}{2})}   2^{\frac{L\beta-J}{2}} \mu^{(d-1)(\frac{1}{p_1'}-\frac{1}{2})}.
 \end{split}
 \eeqq
 Because of $t<1-\frac{1}{p_2}$, we can choose $\beta>0$ small enough such that $-t+1-\frac{1}{p_2}-\frac{\beta}{2}>0 $. Hence
 \beqq
 \begin{split} 
 \sum_{\mu=J+1}^{L} \, 2^{\mu(-t+\frac{1}{p_1}-\frac{1}{p_2})} \, c_n(id_{p_1,p_2-\varepsilon}^{D_{\mu}}) & \lesssim 2^{L(-t+1-\frac{1}{p_2}-\frac{\beta}{2})}   2^{\frac{L\beta-J}{2}} L^{(d-1)(\frac{1}{p_1'}-\frac{1}{2})}\\
 & = 2^{L(-t+1-\frac{1}{p_2} )}   2^{-\frac{J}{2}} L^{(d-1)(\frac{1}{p_1'}-\frac{1}{2})}.
 \end{split}
 \eeqq
We define 
 \be\label{L}
 L=\Big[\frac{p_1'}{2}J + (\frac{p_1'}{2}-1)(d-1)\log J\Big]
 \ee
 which leads to
 \beqq
 \begin{split} 
 \sum_{\mu=J+1}^{L} \, 2^{\mu(-t+\frac{1}{p_1}-\frac{1}{p_2})} \, c_n(id_{p_1,p_2-\varepsilon}^{D_{\mu}}) 
 & \lesssim  2^{\frac{Jp_1'}{2}(-t+\frac{1}{p_1}-\frac{1}{p_2})}J^{(d-1)(\frac{p_1'}{2}-1)(-t+\frac{1}{p_1}-\frac{1}{p_2})}.
 \end{split}
 \eeqq
 Inserting this and \eqref{L} into \eqref{sum1-1} we have found
 \beqq
 c_n(id^*) \lesssim 2^{\frac{Jp_1'}{2}(-t+\frac{1}{p_1}-\frac{1}{p_2})}J^{(d-1)(\frac{p_1'}{2}-1)(-t+\frac{1}{p_1}-\frac{1}{p_2})} 
 \eeqq
for $n\asymp 2^JJ^{d-1}$. Employing monotonicity arguments we finish the proof.
\end{proof}
\begin{remark}\label{app}
\rm Because of
\beqq
a_n(id_{p_1,p_2}^{m})\asymp \max\Big\{m^{\frac{1}{p_2}-\frac{1}{p_1}},\min\{1,m^{1-\frac{1}{p_1}}n^{-\frac{1}{2}}\}  \sqrt{1-n/m}  \Big\}  
\eeqq
if $ 2\leq p_1'<p_2$, see \cite{Glus}, by  similar argument as above, we obtain part (v) in Theorem \ref{known}.
\end{remark}
%&&&&&&&&&&&&&&&&&&&&&&&&&&&&&&&&&&&&&&&&&&&&&&&&&&&&&&&&&&&&&&&&&&&&&&&&&&&&&&&&&&&&&&&&&&&&&&&&&&&&&&&&&&&&&&&&&&&&&&

\begin{proposition}\label{bon} Let $1<p_1<  p_2\leq 2$ and $\frac{1}{p_1}-\frac{1}{p_2}<t< \frac{1/p_1-1/p_2}{2/p_1-1}$. Then
\[ 
c_n(id^* )\asymp n^{-\frac{p_1'}{2}(t-\frac{1}{p_1}+\frac{1}{p_2})}(\log n)^{(d-1)(t-\frac{1}{p_1}+\frac{1}{p_2})} \, , \qquad n\geq 2.
\]
\end{proposition}
\begin{proof}{\it Step 1.} Estimate from below. If $p_2< 2$ we choose $\varepsilon>0$ such that $p_2+\varepsilon\leq 2$ and employ \eqref{case6} to obtain
\beqq
2^{\mu( -t+\frac{1}{p_1} -\frac{1}{p_2})}\, c_n(id_{p_1,p_2+\varepsilon}^{D_\mu}) \lesssim
c_n(id^*).
\eeqq
Now \eqref{wth2} with $n=\big[D_{\mu}^{2/p_1'}\big]$ leads to
\beqq
2^{\mu( -t+\frac{1}{p_1} -\frac{1}{p_2})}  \lesssim
c_n(id^*).
\eeqq
This implies the lower estimate  if $p_2<2$. By using \eqref{case1} and a similar argument we get the result for $p_2=2$ as well.\\
{\it Step 2.} Estimate from above.  Since $p_2\leq 2$ from \eqref{sum1} and \eqref{case1} we arrive at
 \begin{equation}\label{sum2}
 c_n(id^*)\lesssim \sum_{\mu=J+1}^{L} \, 2^{\mu(-t+\frac{1}{p_1}-\frac{1}{p_2})} \, c_n(id_{p_1,p_2}^{D_{\mu}})+ 2^{-L\big(t-\frac{1}{p_1}+\frac{1}{p_2}\big)}.
 \end{equation}
Next we define $n_{\mu}$, $\mu=J+1,...,L$, as in \eqref{n_mu}. Now \eqref{wth2} leads to
 \beqq
 \begin{split} 
 2^{\mu(-t+\frac{1}{p_1}-\frac{1}{p_2})}c_n(id_{p_1,p_2 }^{D_{\mu}}) &\lesssim 2^{\mu(-t+\frac{1}{p_1}-\frac{1}{p_2})} \Big( D_{\mu}^{1-\frac{1}{p_1}} \big[D_{\mu}\, 2^{\{(\mu-L)\beta+J-\mu \}}\big]^{-\frac{1}{2}}\Big)^{\frac{1/p_1-1/p_2}{1/p_1-1/2}} \\
 & = 2^{\mu(-t + \frac{1/p_1-1/p_2}{2/p_1-1 }- \frac{1/p_1-1/p_2}{2/p_1-1 }\beta)}   \Big(2^{\frac{L\beta-J}{2}}\Big)^{\frac{1/p_1-1/p_2}{1/p_1-1/2 }} \mu^{(d-1)(\frac{1}{p_2}-\frac{1}{p_1})}.
 \end{split}
 \eeqq
 Because of $t< \frac{1/p_1-1/p_2}{2/p_1-1}$, we can choose $\beta>0$ small enough such that $$-t + \frac{1/p_1-1/p_2}{2/p_1-1 }- \frac{1/p_1-1/p_2}{2/p_1-1 }\beta>0 .$$ 
 Consequently
 \be \label{sum3}
 \begin{split} 
 \sum_{\mu=J+1}^{L} \, 2^{\mu(-t+\frac{1}{p_1}-\frac{1}{p_2})} \, c_n(id_{p_1,p_2}^{D_{\mu}}) & \lesssim 2^{L(-t + \frac{1/p_1-1/p_2}{2/p_1-1 }- \frac{1/p_1-1/p_2}{2/p_1-1 }\beta)}   \Big(2^{\frac{L\beta-J}{2}}\Big)^{\frac{1/p_1-1/p_2}{1/p_1-1/2 }} J^{(d-1)(\frac{1}{p_2}-\frac{1}{p_1})}\\
& \asymp 2^{L(-t + \frac{1/p_1-1/p_2}{2/p_1-1 } )}    2^{-J\frac{1/p_1-1/p_2}{2/p_1-1}} J^{(d-1)(\frac{1}{p_2}-\frac{1}{p_1})}
\\
& \asymp 2^{-tL  }    2^{(L-J)\frac{1/p_1-1/p_2}{2/p_1-1}} J^{(d-1)(\frac{1}{p_2}-\frac{1}{p_1})}.
 \end{split}
 \ee 
 Again we define 
 \beqq
 L:=\Big[\frac{p_1'}{2}J + (\frac{p_1'}{2}-1)(d-1)\log J\Big] \asymp \frac{p_1'}{2}\Big(\frac{2}{p_1}-1 \Big)J + \frac{p_1'}{2} \Big( \frac{2}{p_1}-1\Big)(d-1)\log J +J.
 \eeqq
Inserting this into \eqref{sum3} we find
 \beqq
 \begin{split} 
 \sum_{\mu=J+1}^{L} \, 2^{\mu(-t+\frac{1}{p_1}-\frac{1}{p_2})} \, c_n(id_{p_1,p_2}^{D_{\mu}}) 
& \lesssim 2^{-t\frac{p_1'}{2}  }  J^{(d-1)t(1-\frac{p_1'}{2})}  2^{\frac{p_1'}{2}(J+(d-1)\log J)(\frac{1}{p_1}-\frac{1}{p_2})} J^{(d-1)(\frac{1}{p_2}-\frac{1}{p_1})} \\
& = 2^{-J\frac{p_1'}{2}(t-\frac{1}{p_1}+\frac{1}{p_2})}J^{(d-1)(1-\frac{p_1'}{2})(t-\frac{1}{p_1}+\frac{1}{p_2})}.
 \end{split}
 \eeqq
This together with the special choice of $L$ leads to
\beqq
c_n(id^*) \lesssim 2^{-J\frac{p_1'}{2}(t-\frac{1}{p_1}+\frac{1}{p_2})}J^{(d-1)(1-\frac{p_1'}{2})(t-\frac{1}{p_1}+\frac{1}{p_2})}
\eeqq
for $n\asymp 2^JJ^{d-1}$, see \eqref{sum2}. Finally, we finish the proof by the standard monotonicity argument.
\end{proof}
\begin{proposition}\label{nam}
 Let $1<p_1,p_2<2$ and $ \max( 0, \frac{1/p_1-1/p_2}{2/p_1-1})<t<\frac{1}{2}$. Then we have
\beqq
c_n(id^* )\lesssim n^{-t}(\log n)^{(d-1)(2t-\frac{2t}{p_1})},\ \ n\geq 2.
\eeqq 
\end{proposition}
\begin{proof}
 {\it Step 1.} The case $p_1<p_2<2$ and $\frac{1/p_1-1/p_2}{2/p_1-1}<t<\frac{1}{2}$. We split the sum in \eqref{sum1} into two terms
\be \label{sum_5}
\begin{split} 
 c_n(id^*)&\lesssim \sum_{\mu=J+1}^{K}c_{n_{\mu}}(id_{\mu}^*) +\sum_{\mu=K+1}^{L} \, c_{n_{\mu}}(id_{\mu}^*)+ 2^{-L\big(t- \frac{1}{p_1}+\frac{1}{p_2} \big)} \\
 & \lesssim \sum_{\mu=J+1}^{K}2^{\mu(-t+\frac{1}{p_1}-\frac{1}{2})}   c_n(id_{p_1,2}^{D_{\mu}}) +\sum_{\mu=K+1}^{L} 2^{\mu(-t+\frac{1}{p_1}-\frac{1}{p_2})}   c_n(id_{p_1,p_2}^{D_{\mu}})+ 2^{-L\big(t- \frac{1}{p_1}+\frac{1}{p_2} \big)},
 \end{split}
\ee 
see \eqref{case1}. We define 
$$K=\Big[J+\Big(\frac{2}{p_1}-1\Big)(d-1)\log J\Big]$$
and
\[
n_{\mu} : =
\begin{cases}
\big[D_{\mu} \, 2^{(\mu-K)\beta+J-\mu}\big] & \qquad \text{if}\quad J+1 \le \mu \leq K\, , \\
\big[J^{d-1}2^{J}\, 2^{(K-\mu)\gamma}\big] & \qquad \text{if}\quad K+1 \le \mu \leq L\, . 
\end{cases}
\]
Here $\beta, \, \gamma>0$ will be fixed later. The condition $\beta, \, \gamma>0$ implies \eqref{2^J}. We estimate the first sum on the right-hand side of \eqref{sum_5}. We have
\beqq
 \begin{split} 
  \sum_{\mu=J+1}^{K}2^{\mu(-t+\frac{1}{p_1}-\frac{1}{2})}   c_n(id_{p_1,2}^{D_{\mu}})&\lesssim \sum_{\mu=J+1}^{K} \, 2^{\mu(-t+\frac{1}{p_1}-\frac{1}{2})} \,   D_{\mu}^{1-\frac{1}{p_1}}\big[D_{\mu}\, 2^{\{(\mu-K)\beta+J-\mu \}}\big]^{-\frac{1}{2}} \\
  & \asymp \sum_{\mu=J+1}^{K} \, 2^{\mu(-t+\frac{1}{2}-\frac{\beta}{2} )} \,   2^{\frac{K\beta-J}{2}} \mu^{(d-1)(\frac{1}{2}-\frac{1}{p_1})} ,
 \end{split}
 \eeqq
see \eqref{wth2}. Since $t<\frac{1}{2}$ we can choose $\beta>0$ such that $-t+\frac{1}{2}-\frac{\beta}{2}>0$. Consequently we obtain
  \be \label{J-K}
  \begin{split} 
  \sum_{\mu=J+1}^{K}2^{\mu(-t+\frac{1}{p_1}-\frac{1}{2})}   c_n(id_{p_1,2}^{D_{\mu}})&\lesssim   2^{K(-t+\frac{1}{2}-\frac{\beta}{2} )} \,   2^{\frac{K\beta-J}{2}} K^{(d-1)(\frac{1}{2}-\frac{1}{p_1})}  \\
   &\asymp 2^{-Kt}2^{\frac{K-J}{2}}K^{(d-1)(\frac{1}{2}-\frac{1}{p_1})} .
  \end{split}
  \ee
 Now we deal with the second sum on the right-hand side of \eqref{sum_5}. From \eqref{wth2} we have 
\beqq
\begin{split} 
\sum_{\mu=K+1}^{L} 2^{\mu(-t+\frac{1}{p_1}-\frac{1}{p_2})}   &c_n(id_{p_1,p_2}^{D_{\mu}}) \lesssim \sum_{\mu=K+1}^{L} 2^{\mu(-t+\frac{1}{p_1}-\frac{1}{p_2})}\Big( D_{\mu}^{1-\frac{1}{p_1}} \big[J^{d-1}2^{J}\, 2^{(K-\mu)\gamma}\big]^{-\frac{1}{2}}\Big)^{\frac{1/p_1-1/p_2}{1/p_1-1/2}} \\
& \asymp \sum_{\mu=K+1}^{L} 2^{\mu(-t + \frac{1/p_1-1/p_2}{2/p_1-1 }- \frac{1/p_1-1/p_2}{2/p_1-1 }\gamma)}   \Big(\mu^{(d-1)(1-\frac{1}{p_1})} \big[J^{d-1}2^{J}\, 2^{K\gamma}\big]^{-\frac{1}{2}}\Big)^{\frac{1/p_1-1/p_2}{1/p_1-1/2 }}.
\end{split}
\eeqq
Since $t> \frac{1/p_1-1/p_2}{2/p_1-1}$ we can choose $\gamma>0$ such that $-t + \frac{1/p_1-1/p_2}{2/p_1-1 }- \frac{1/p_1-1/p_2}{2/p_1-1 }\gamma<0$. This leads to
\be \label{K-L}
\begin{split} 
\sum_{\mu=K+1}^{L} 2^{\mu(-t+\frac{1}{p_1}-\frac{1}{p_2})}   c_n(id_{p_1,p_2}^{D_{\mu}}) &\lesssim 2^{K(-t + \frac{1/p_1-1/p_2}{2/p_1-1 }- \frac{1/p_1-1/p_2}{2/p_1-1 }\gamma)}   \Big(K^{(d-1)(1-\frac{1}{p_1})} \big[J^{d-1}2^{J}\, 2^{K\gamma}\big]^{-\frac{1}{2}}\Big)^{\frac{1/p_1-1/p_2}{1/p_1-1/2 }}\\
& \asymp 2^{-Kt} \Big(K^{(d-1)(1-\frac{1}{p_1})} J^{-\frac{d-1}{2}}  2^{\frac{K-J}{2}}\Big)^{\frac{1/p_1-1/p_2}{1/p_1-1/2 }}\\
&\asymp 2^{-Kt} \Big(J^{(d-1)(\frac{1}{2}-\frac{1}{p_1})}   2^{\frac{K-J}{2}}\Big)^{\frac{1/p_1-1/p_2}{1/p_1-1/2 }}.
\end{split}
\ee 
The last line is due to $J<K< dJ$. Replacing $K$ into \eqref{J-K} and \eqref{K-L} we arrive at
\beqq
\sum_{\mu=J+1}^{L} c_{n_{\mu}}(id_{\mu}^*) \lesssim 2^{-Kt}\asymp   2^{-Jt}J^{(d-1)(t-\frac{2t}{p_1})}.
\eeqq
Choosing $L$ large enough we have proved 
\beqq
c_n(id^*) \lesssim 2^{-Jt}J^{(d-1)(t-\frac{2t}{p_1})}
\eeqq
for $n\asymp 2^JJ^{d-1}$. By monotonicity of Gelfand numbers we finish the proof in this case.\\
{\it Step 2.} The case $p_2\leq p_1<2$ and $0<t<\frac{1}{2}$. Since $t>0$ we can choose $p_1<p<2$ such that $\frac{1/p_1-1/p}{2/p_1-1}<t<\frac{1}{2}$. We consider the chain of embeddings
\beqq
 s_{p_1,p_1}^{t,\Omega}f\hookrightarrow s_{p,2}^{0,\Omega}f\hookrightarrow s_{p_2,2}^{0,\Omega}f
\eeqq
and obtain 
\beqq
c_n(id^*) \leq c_n(id: s_{p_1,p_1}^{t,\Omega}f\to s_{p,2}^{0,\Omega}f) \cdot \| id : s_{p,2}^{0,\Omega}f\to s_{p_2,2}^{0,\Omega}f\| \lesssim c_n(id: s_{p_1,p_1}^{t,\Omega}f\to s_{p,2}^{0,\Omega}f),
\eeqq
see property (c) of the $s$-numbers. Finally the result in Step 1 implies the desired estimate. The proof is complete.
\end{proof}
\subsection{Proof of the main results}
We are now in position to prove Theorems \ref{main}, \ref{ex}, \ref{ex-3} and Proposition \ref{conj}.
\vskip 3mm
\noindent
{\bf Proof of Theorem \ref{main}.} The cases $\max(2,p_2)\leq p_1$ and (iii), (iv) are consequences of Lemma \ref{weyl} and Propositions \ref{mot}$-$\ref{ba} and \ref{bon}. The lower bounds of the cases $p_1,p_2\leq 2$ and (ii) follow from the relation $x_n\leq c_n$. We refer to \cite{KiSi} for the asymptotic behaviour of $x_n({\rm App}: S^t_{p_1,p_1}B(\Omega) \to L_{p_2}(\Omega))$. The upper bound of (ii) is derived from the inequality $c_n\leq a_n$ and part (ii) in Theorem \ref{known}. To finish we consider the chain of continuous embeddings
\beqq
S^t_{p_1,p_1}B(\Omega) \hookrightarrow L_2(\Omega) \hookrightarrow L_{p_2}(\Omega)
\eeqq  
if $p_2\leq 2$. Now property (c) of the $s$-numbers together with part (iii) in Theorem \ref{main} implies the estimate from above in the case  $p_1,p_2<2$ and $t>\frac{1}{2}$. The proof is complete. \qed

\vskip 3mm
\noindent
{\bf Proof of Proposition \ref{conj}.} The proof follows from Lemma \ref{weyl}  in combination with Proposition \ref{nam}. \qed
\vskip 3mm

The following proposition will be used to prove the results in Theorems \ref{ex} and \ref{ex-3}, see \cite{Tr70,KiSi}.
\begin{proposition}\label{inter}
Let $0<\theta<1$.
Let $X,Y, Y_0,Y_1$ be Banach spaces. 
Further we assume  $Y_0\cap Y_1\hookrightarrow  Y$ and the existence of  a positive constant $C$ such that
\beqq
\| y|Y\| \leq C \, \| y|Y_0\|^{1-\theta}\|y|Y_1\|^{\theta}\qquad  \text{for all}\quad  y\in Y_0\cap Y_1.
\eeqq
Then, if $
T\in  \mathcal{L}(X,Y_0) \cap  \mathcal{L}(X,Y_1) \cap  \cl (X,Y)
$
we obtain 
\beqq
c_{n+m-1}(T:~X\to Y)\le C\,  c_n^{1-\theta}(T:~ X\to Y_0)\, c_m^{\theta}(T:~X\to Y_1)
\eeqq
for all $n,m \in \N$. 
\end{proposition}

\noindent
{\bf Proof of Theorem \ref{ex}.} {\it Step 1.} We prove (i). \\{\it Substep 1.1.} Estimate from above. Since $p>1$ there exist $ \varepsilon>0$ such that $1+\epsilon <p$. We have the following embeddings
\beqq
S^t_{p,p}B(\Omega )\hookrightarrow L_{1+\varepsilon}(\Omega) \hookrightarrow L_1(\Omega).
\eeqq
Property (c) of the $s$-numbers together with the result in Theorem \ref{main} (i) implies the estimate from above.\\
{\it Substep 1.2.} Estimate from below. Since $p>1$ there exist $p_0,p_1$ and $0<\theta<1$ such that $$1<p_0<p_1<\min(p,2) \qquad \text{and}\qquad \frac{1}{p_0}=\frac{\Theta}{1}+\frac{1-\Theta}{p_1}.$$
This yields 
\[
\| f|L_{p_0}(\Omega)\| \leq \|f|L_1(\Omega)\|^{1-\theta}\, \|f|L_{p_1}(\Omega)\|^{\theta}\qquad   \text{for all}\quad  f\in L_{p_1}(\Omega).
 \]
 Next we employ the interpolation property of the Gelfand numbers, see Proposition \ref{inter}, and obtain
 \beqq
 c_{2n-1}  ({\rm App} :  S^t_{p,p}B(\Omega)\to L_{p_0}(\Omega))\lesssim  
   c_n^{1-\theta}({\rm App}: S^t_{p,p}B(\Omega)\to L_1(\Omega))\, \, c_n^{\theta}({\rm App}:S^t_{p,p}B(\Omega)\to L_{p_1}(\Omega)).
 \eeqq
 Now, the estimate from below follows from part (i) in Theorem \ref{main}. \\
 {\it Step 2.} Proof of (ii). The lower estimate follow from the inequality $x_n\leq c_n$. We refer again to \cite[Theorem 3.4]{KiSi} for asymptotic behaviour of $x_n({\rm App}:\ S^t_{p,p}B(\Omega) \to L_{\infty}(\Omega))$. Let $1<p<2$ and $t>1$. Then there always exists some $r>\frac{1}{2}$ such that $t-r>\frac{1}{2}$.
 We consider the commutative diagram
 \tikzset{node distance=4cm, auto}
 \begin{center}
 \begin{tikzpicture}
  \node (H) {$S^{t}_{p,p}B(\Omega)$};
  \node (L) [right of =H] {$L_\infty(\Omega)$};
  \node (L2) [right of =H, below of =H, node distance = 2cm ] {$S^{r}_{2,2}B(\Omega)$};
  \draw[->] (H) to node {${\rm App}$} (L);
  \draw[->] (H) to node [swap] {${\rm App}_1$} (L2);
  \draw[->] (L2) to node [swap] {${\rm App}_2$} (L);
  \end{tikzpicture}
 \end{center}
The multiplicativity of the Gelfand numbers, see \cite[Section 11.9]{Pi-80}, yields 
\be \label{inf}
c_{2n-1} ({\rm App})\le c_n ({\rm App}_1)\, c_n ({\rm App}_2)\asymp c_n({\rm App}_1) \cdot n^{-r+\frac{1}{2}}(\log n)^{(d-1)r}\, ,
\ee
see Proposition \ref{known1}. By the lifting property  of mixed Besov spaces, see Theorem \ref{lift}, and Theorem \ref{main} (iii) we have
\beqq
c_n({\rm App}_1)\asymp c_n({\rm App}:S^{t-r}_{p,p}B(\Omega)\to L_2(\Omega) )\asymp n^{-t+r}(\log n)^{(d-1)(t-r-\frac{1}{p}+\frac{1}{2})}.
\eeqq
Putting this into \eqref{inf} we get the desired upper estimate. The proof is complete.
\qed
\vskip 3mm
\noindent
{\bf Proof of Theorem \ref{ex-3}.} {\it Step 1.} Proof of (i). Recall that 
\beqq
a_n({\rm App}:\ S^t_{p}H(\Omega) \to L_{1}(\Omega))\asymp n^{-t}(\log n)^{(d-1)t}, \qquad n \ge2
\eeqq was obtained by Romanyuk \cite{Rom9}. From this and the inequality $c_n\leq a_n$ we get the upper bound for Gelfand numbers. By similar arguments as in the Substep 1.2 of the proof of Theorem \ref{ex} and the result in part (i) in Theorem \ref{gel-so} we obtain the estimate from below as well. \\
{\it Step 2.} Proof of (ii). Since the target space is $L_{\infty}(\Omega)$, it is enough to prove (ii) for Gelfand numbers. The lower estimate is a consequence of the inequality $x_n\leq c_n$ and the result in \cite[Theorem 2.6]{Ki16}. Concerning the estimate from above we consider the diagram 
 \tikzset{node distance=4cm, auto}
 \begin{center}
 \begin{tikzpicture}
  \node (H) {$S^{t}_{p}H(\Omega)$};
  \node (L) [right of =H] {$L_\infty(\Omega)$};
  \node (L2) [right of =H, below of =H, node distance = 2cm ] {$S^{r}_{2}H(\Omega)$};
  \draw[->] (H) to node {${\rm App}$} (L);
  \draw[->] (H) to node [swap] {${\rm App}_1$} (L2);
  \draw[->] (L2) to node [swap] {${\rm App}_2$} (L);
  \end{tikzpicture}
 \end{center}
 Now similar arguments as in Step 2 of the proof of Theorem \ref{ex} yields the desired result. This finishes the proof. \qed
\vskip 3mm
\noindent
{\bf Acknowledgements:} The author would like to thank Professor Winfried Sickel for many valuable discussions and comments about this work.


\begin{thebibliography}{9999}


%&&&&&&&&&&&&&&&&&&&&&&&&&&&&&&&&&&&&&&&&&&&&&&&&&&&&&&&&&&&&&&&&&&&&&&&&&&&&&&&&&&&&&&&&&&&&&&&&&&&&&


\bibitem{Am}
{\sc T.I.~Amanov}, Spaces of differentiable functions with dominating mixed derivatives,
Nauka Kaz. SSR, Alma-Ata, 1976.





\bibitem{Baz1}
{\sc D.B.~Bazarkhanov}, Characterizations of Nikol'skij-Besov
and Lizorkin-Triebel function spaces of mixed smoothness, Proc. Steklov Inst. {\bf 243} (2003), 46-58.

\bibitem{Baz2}
{\sc D.B.~Bazarkhanov}, Equivalent (quasi)normings of some function spaces of generalized mixed smoothness,
Proc. Steklov Inst. {\bf 248} (2005), 21-34.

\bibitem{Baz3}
{\sc D.B.~Bazarkhanov}, Wavelet representations and equivalent normings of some function spaces of generalized mixed smoothness, 
Math. Zh. {\bf 5} (2005), 12-16.




\bibitem{Baz7}
{\sc D.B.~Bazarkhanov},
Estimates for widths of classes of periodic functions of several variables - I,
Eurasian Math. J. {\bf 1} (2010), 11-26.


\bibitem{CKS}
{\sc F.~Cobos, T.~K\"{u}hn, W.~Sickel}, Optimal approximation of multivariate periodic Sobolev functions in the sup-norm, J. Funct.
Anal. {\bf 270} (2016), 4196-4212.

%\bibitem{Be}
%{\sc E.S.~Belinsky}, Estimates of entropy numbers and Gaussian measures for classes of functions with bounded mixed derivative.
%{\it JAT} {\bf 93} (1998), 114-127.


%\bibitem{BL}
%{\sc J.~Bergh and J.~L\"ofstr\"om}, {\em Interpolation Spaces. An Introduction.} Springer, New York, 1976.
%
%
% 

 



\bibitem{DF}
{\sc A.~Defant, K.~Floret}, Tensor norms and operator ideals, North Holland, Amsterdam, 1993.

%\bibitem{DD-T-T}
%{\sc Dinh~D\~ung, V.N. Temlyakov and T. Ullrich}, Hyperbolic cross approximation, {\it Preprint}, arXiv:1601.03978.


\bibitem{EL}
{\sc D.E.~Edmunds, J.~Lang}, {Gelfand numbers and widths}, J. Approx. Theory {\bf 166} (2013), 78-84.


\bibitem{DTU-16}
{\sc Dinh~D\~ung, V.N. Temlyakov, T. Ullrich}, Hyperbolic cross approximation, Preprint, arXiv:1601.03978.

\bibitem{Ga1}
{\sc E.M.~Galeev}, Widths of the Besov classes $B^r_{p,\theta}(\mathbb{T}^d)$, Math. Notes {\bf 69} (2001), 605-613. 

%\bibitem{Ga2}
%{\sc E.M.~Galeev}, Approximation of classes of periodic functions of several variables by nuclear operators. {\it Math. Notes} 
%{\bf 47} (1990), 248-254.

\bibitem{Glas-04}
{\sc P. Glasserman}, Monte Carlo Methods in Financial Engineering. Applications of mathematics: stochastic
modeling and applied probability, Springer, 2004.

\bibitem{Glus}
{\sc E.D.~Gluskin},
Norms of random matrices and widths of finite-dimensional sets, Math. USSR Sb. {\bf 48} (1984) 173-182.





\bibitem{Hansen}
{\sc M.~Hansen}, Nonlinear approximation and function spaces of dominating mixed smoothness, PhD thesis, 
Friedrich-Schiller-University Jena, 2010.












\bibitem{LiCh}
{\sc W.A. Light, E.W. Cheney}, Approximation theory in
tensor product spaces, Lecture Notes in Math. {\bf 1169}, Springer,
Berlin, 1985.




%
%
%\bibitem{lun}
%{\sc A.~Lunardi}, {\it Interpolation theory}. 
%Lect. Notes, Scuola Normale Superiore Pisa, 2009.
%



\bibitem{KiSi}
{\sc V.K. Nguyen, W. Sickel}, Weyl numbers of embeddings of tensor product Besov spaces, J. Approx. Theory {\bf 200} (2015), 170-220.

\bibitem{Ki}
{\sc V.K. Nguyen},
Bernstein numbers of embeddings of isotropic and dominating mixed Besov spaces, Math. Nachr.  {\bf 288} (2015), 1694-1717.

\bibitem{Ki16}
{\sc V.K. Nguyen}, Weyl and Bernstein numbers of embeddings of Sobolev spaces with dominating mixed smoothness, J. Complexity (2016), accepted.

\bibitem{Ni} 
{\sc S. M. Nikol'skij}, Approximation of functions of several variables
and embedding theorems, Springer, Berlin, 1975.


\bibitem{NoWo-1}
{\sc E. Novak, H. Wo\'zniakowski}, Tractability of Multivariate Problems, Volume I: Linear Information, European Math. Soc., Z\"urich, 2008.

%\bibitem{NoWo-2}
%{\sc E. Novak and H. Wo\'zniakowski}, Tractability of Multivariate Problems, Volume II: Standard
%Information for Functionals, European Math. Soc. Publ. House, Z\"urich, 2010.
%
%\bibitem{NoWo-3}
%{\sc E. Novak and H. Wo\'zniakowski}, Tractability of Multivariate Problems, Volume III: Standard
%Information for Operators, European Math. Soc. Publ. House, Z\"urich, 2012.

\bibitem{Pi-80}
{\sc A. Pietsch}, Operator Ideals,
North-Holland, Amsterdam, 1980.




\bibitem{Pi-87}
{\sc A. Pietsch}, Eigenvalues and $s$-numbers,
Cambridge University Press, Cambridge, 1987.

\bibitem{Pin85} 
{\sc A. Pinkus}, $n$-Widths in Approximation Theory, Springer-Verlag, Berlin, 1985.

\bibitem{Rom6}
{\sc A.S.~Romanyuk}, Approximation of the Besov classes of periodic functions of several variables in a space $L_q$,
Ukrainian Math. J. {\bf 43} (1991), 1297-1306. 



\bibitem{Rom7}
{\sc A.S.~Romanyuk}, The best trigonometric approximations and the Kolmogorov diameters of the Besov classes of functions of many variables,
Ukrainian Math. J. {\bf 45} (1993), 724-738.

\bibitem{Rom4}
{\sc A.S.~Romanyuk}, On Kolmogorov widths of classes $B^r_{p,\theta}$ of periodic 
function of many variables with low smoothness in the space $L_q$,
Ukrainian Math. J. {\bf 46} (1994), 915-926.



\bibitem{Rom8}
{\sc A.S.~Romanyuk}, On the best approximations and Kolmogorov widths of the Besov classes of periodic functions of many variables,
Ukrainian Math. J. {\bf 47} (1995), 91-106. 


\bibitem{Rom1}
{\sc A.S.~Romanyuk}, Linear widths of the Besov classes of periodic functions of many variables. I,
Ukrainian Math. J. {\bf 53} (2001), 647-661.



\bibitem{Rom2}
{\sc A.S.~Romanyuk}, Linear widths of the Besov classes of periodic functions of many variables. II, 
Ukrainian Math. J. {\bf 53} (2001), 820-829.

\bibitem{Rom5}
{\sc A.S.~Romanyuk}, 
On estimates of the Kolmogorov widths of the classes $B^r_{p,\theta}$ in the space $L_q$,
Ukrainian Math. J. {\bf 53} (2001), 996-1001.

\bibitem{Rom9}
{\sc A.S.~Romanyuk}, Best approximations and widths of classes of periodic functions of many variables, 
Math. Sbornik {\bf 199} (2008), 253-275.

%\bibitem{Rom3}
%{\sc A.S.~Romanyuk}, Kolmogorov widths of the Besov classes $B^r_{p,\theta}$ in the metric of the space $L_\infty$. {\it Ukr. Mat. Visn.} {\bf 2}(2) (2005), 201-218.
%
%
%%
%\bibitem{Rom9}
%{\sc A.S.~Romanyuk}, Best approximations and widths of classes of periodic functions of many variables. 
%{\it Math. Sbornik} {\bf 199} (2008), 253-275.
%


%\bibitem{Sc2}
%{\sc H.-J. Schmeisser},
%Recent developments in the theory of function spaces with dominating mixed smoothness. In: {\it 
%Proc. Conf. NAFSA-8}, Prague 2006, (ed. J.~Rakosnik),
%Inst. of Math. Acad. Sci., Czech Republic, Prague, 2007, pp.~145-204.


\bibitem{ST} 
{\sc H.-J. Schmeisser, H. Triebel}, Topics in Fourier
analysis and function spaces,
 Geest \& Portig, Leipzig, 1987 and Wiley, Chichester, 1987.





\bibitem{SUt}
{\sc W.~Sickel, T.~Ullrich}, 
Tensor products of Sobolev-Besov spaces and applications to approximation from the hyperbolic cross,
J. Approx. Theory {\bf 161} (2009), 748-786.


\bibitem{SUspline}
{\sc W.~Sickel, T.~Ullrich},
Spline interpolation on sparse grids.
Applicable Analysis {\bf 90} (2011), 337-383.


%\bibitem{Tem2}
%{\sc V.N. Temlyakov}, The estimates of asymptotic characteristics on functional classes with bounded mixed derivative or difference. 
%{\it Trudy Mat. Inst. Steklov.} {\bf 189} (1989), 138-167.

\bibitem{TrWaWo}
{\sc J.F. Traub, G.W. Wasilkowski, H. Wo\'zniakowski}, Information-Based
Complexity, Academic Press, New York, 1988.


\bibitem{Tem}
{\sc V.N. Temlyakov}, Approximation of periodic functions,
Nova Science, New York, 1993.

\bibitem{Te93}
{\sc V.N. Temlyakov}, On approximate recovery of functions with bounded mixed derivative,
J. Complexity {\bf 9} (1993), 41-59.



%\bibitem{Te93}
%{\sc V.N. Temlyakov}, On approximate recovery of functions with bounded mixed derivative.
%{\it J. Complexity} {\bf 9} (1993), 41--59.


%
%\bibitem{Te96}
%{\sc V.N. Temlyakov}, An inequality for trigonometric polynomials and its application for estimating the Kolmogorov widths.
%{\it East J. on Approximations} {\bf 2} (1996), 253--262.


\bibitem{Tr70}
{\sc H.~Triebel}, 
Interpolationseigenschaften von Entropie und Durchmesseridealen kompakter Operatoren, 
Studia Math. {\bf 34} (1970), 89-107.
%


\bibitem{Tr83}
{\sc H.~Triebel}, Theory of function spaces, Birkh\"auser,
Basel, 1983.


\bibitem{Tr92}
{\sc H.~Triebel}, Theory of function spaces II, Birkh\"auser,
Basel, 1992.

\bibitem{Tr06}
{\sc H.~Triebel}, Theory of function spaces III, Birkh\"auser,
Basel, 2006.


\bibitem{Tr10}
{\sc H.~Triebel}, Bases in function spaces, sampling, discrepancy, numerical integration, European
Math. Soc. Publishing House,  Z\"urich,  2010.




\bibitem{U1}
 {\sc T.~Ullrich},
 Function spaces with dominating mixed smoothness. Characterizations by differences,
 Jenaer Schriften zur Mathematik und Informatik, Math/Inf/05/06,
 Jena, 2006.
 


\bibitem{Vy-06}
{\sc J.~Vybiral}, Function spaces with dominating mixed smoothness, Dissertationes Math. {\bf 436} (2006).


\bibitem{Vy-08}
{\sc J. Vybiral}, Widths of embeddings in function spaces, J. Complexity {\bf 24} (2008), 545-570.


\bibitem{woj} {\sc P. Wojtaszczyk}, A mathematical introduction to wavelets,
Cambridge Univ. Press, Cambridge, 1997.


\bibitem{Yse-10}
{\sc H. Yserentant}, Regularity and Approximability of Electronic Wave Functions, Lecture
 Notes in Mathematics, Springer, 2010.




\bibitem{ZF-12}
{\sc S.~Zhang, G.~Fang}, Gelfand and Kolmogorov numbers of Sobolev embeddings of weighted function spaces, J. Complexity  {\bf 28}  (2012),   209-223.

\bibitem{ZF-13}
{\sc S.~Zhang, G.~Fang}, Widths of embeddings in weighted function spaces, J. Approx. Theory {\bf 167} (2013), 147-172.




\bibitem{ZD-14}
{\sc S.~Zhang, A. Dota}, $s$-numbers of compact embeddings of function spaces on quasi-bounded domains, J. Complexity  {\bf 30}  (2014), 495-513.


\bibitem{ZFH-14}
{\sc S.~Zhang, G.~Fang, F.~Huang}, 
Some $s$-numbers of embeddings in function spaces with polynomial weights, J. Complexity {\bf 30} (2014), 514-532. 
%\bibitem{Baz7}
%{\sc D.B.~Bazarkhanov},
%Estimates for widths of classes of periodic multivariable functions.
%{\it Doklady Academii Nauk} {\bf 436}(5) (2011), 583 - 585 (russian), engl. transl. in
%{\it Doklady Math.} {\bf 83}(1) (2011), 90-92.



 
%
%\bibitem{Carl}
%{\sc B.~Carl}, Eigenwertverteilungen  von  Operatoren  in  Banachr\"aumen  und  $L_p$-R\"aumen,  Habilitationsschrift, Jena, 1977. 
 



%\bibitem{Di11}
%{\sc Dinh~D\~ung}, B-spline quasi-interpolant representations and sampling recovery of functions with mixed smoothness. {\it J. Complexity} {\bf 27} (2011), 541-567.


%\bibitem{DF}
%{\sc A.~Defant and K.~Floret}, {\it Tensor norms and operator ideals}. North Holland, Amsterdam, 1993.



%\bibitem{EL}
%{\sc D.E.~Edmunds and J.~Lang}, {Gelfand numbers and widths}. {\it JAT} {\bf 166} (2013), 78-84.



%\bibitem{Ga2}
%{\sc E.M.~Galeev}, Approximation of classes of periodic functions of several variables by nuclear operators. {\it Math. Notes} 
%{\bf 47} (1990), 248-254.


%\bibitem{Ga}
%{\sc E.M.~Galeev},
%Linear widths of H\"older-Nikol'skii classes of periodic functions 
%of several variables.
%{\it Mat. Zametki} {\bf 59} No. 2 (1996), 189-199 (russian), engl. transl. in {\it Math. Notes} {\bf 59} (1996), No. 2, 133-140.

%\bibitem{Ga1}
%{\sc E.M.~Galeev}, Widths of the Besov classes $B^r_{p,\theta}(\mathbb{T}^d)$. {\it Math. Notes} {\bf 69}, No. 5, (2001), 605-613. 

%{\sc E.D.~Gluskin},
%Norms of random matrices and widths of finite-dimensional sets. {\it Math. USSR Sb.} {\bf 48} (1984) 173-182.




%\bibitem{Ha}
%{\sc M.~Hansen}, On tensor products of quasi-Banach spaces.
%{\it Preprint} {\bf 63}, DFG-SPP 1324, Marburg, 2010.



%\bibitem{HV}
%{\sc M.~Hansen and J.~Vybiral}, The Jawerth-Franke embedding of spaces of dominating mixed smoothness.
%{\it Georg.~J.~Math.} {\bf 16}(4) (2009), 667-682.



%\bibitem{LiCh}
%{\sc W.A. Light and E.W. Cheney}, \emph{Approximation theory in
%tensor product spaces}. Lecture Notes in Math. {\bf 1169}, Springer, Berlin, 1985.




%\bibitem{Nit}
%{\sc P.-A. Nitsche}, Best N-term approximation spaces for tensor product wavelet bases. 
%{\it Constr. Approx.} {\bf 24} (2006), 49-70.



%\bibitem{Rom6}
%{\sc A.S.~Romanyuk}, Approximation of the Besov classes of periodic functions of several variables in a space $L_q$.
%{\it Ukrainian Math. J.} {\bf 43}(10) (1991), 1297-1306. 



%\bibitem{Rom7}
%{\sc A.S.~Romanyuk}, The best trigonometric approximations and the Kolmogorov diameters of the Besov classes of functions of many variables.
%{\it Ukrainian Math. J.} {\bf 45} (1993), 724-738.

%\bibitem{Rom4}
%{\sc A.S.~Romanyuk}, On Kolmogorov widths of classes $B^r_{p,\theta}$ of periodic 
%function of many variables with low smoothness in the space $L_q$.
%{\it Ukrainian Math. J.} {\bf 46} (1994), 915-926.



 %bibitem{Rom8}
%{\sc A.S.~Romanyuk}, On the best approximations and Kolmogorov widths of the Besov classes of periodic functions of many variables.
%{\it Ukrainian Math. J.} {\bf 47} (1995), 91-106. 


%\bibitem{Rom1}
%{\sc A.S.~Romanyuk}, Linear widths of the Besov classes of periodic functions of many variables. I. 
%{\it Ukrainian Math. J.} {\bf 53} (2001), 647-661.



%\bibitem{Rom2}
%{\sc A.S.~Romanyuk}, Linear widths of the Besov classes of periodic functions of many variables. II. 
%{\it Ukrainian Math. J.} {\bf 53} (2001), 820-829.

%\bibitem{Rom5}
%{\sc A.S.~Romanyuk}, 
%On estimates of the Kolmogorov widths of the classes $B^r_{p,\theta}$ in the space $L_q$.
%{\it Ukrainian Math. J.} {\bf 53} (2001), 996-1001.


%\bibitem{Rom3}
%{\sc A.S.~Romanyuk}, Kolmogorov widths of the Besov classes $B^r_{p,\theta}$ in the metric of the
%space $L_\infty$. {\it Ukr. Mat. Visn.} {\bf 2}(2) (2005), 201-218.




%\bibitem{Rom9}
%{\sc A.S.~Romanyuk}, Best approximations and widths of classes of periodic functions of many variables. 
%{\it Math. Sbornik} {\bf 199} (2008), 253-275.




%\bibitem{Sc2}
%{\sc H.-J. Schmeisser},
%Recent developments in the theory of function spaces with dominating mixed smoothness. In: {\it 
%Proc. Conf. NAFSA-8}, Prague 2006, (ed. J.~Rakosnik),
%Inst. of Math. Acad. Sci., Czech Republic, Prague, 2007, pp.~145-204.






%\bibitem{SUspline}
%{\sc W.~Sickel and T.~Ullrich},
%Spline interpolation on sparse grids.
%{\it Applicable Analysis} {\bf 90} (2011), 337-383.



%\bibitem{Spreng}
%{\sc F. Sprengel},
%A tool for approximation in bivariate periodic Sobolev spaces.
%In: {\it Approximation Theory IX}, Vol. {\bf 2}, Vanderbilt Univ. Press, Nashville (1999), 319-326.






%\bibitem{Tem}
%{\sc V.N. Temlyakov}, \emph{Approximation of periodic functions}.
%Nova Science, New York, 1993.



 






%\bibitem{Tr92}
%{\sc H.~Triebel}, {\it Theory of function spaces II}. Birkh\"auser,
%Basel, 1992.

%\bibitem{Tr06}
%{\sc H.~Triebel}, {\it Theory of function spaces III}. Birkh\"auser,
%Basel, 2006.




%\bibitem{Tu}
%{\sc P.~Turpin},
%Repr{\'e}sentation fonctionelle des espace vectorielles topologiques.
%{\it Studia Math.} {\bf 73} (1982), 1-10.

%\bibitem{Vy}
%{\sc J. Vybiral}, Widths of embeddings in function spaces. {\it J. Complexity} {\bf 24} (2008), 545-570.


\end{thebibliography}
\end{document}